\documentclass[10pt]{amsart}

\usepackage[a4paper,margin=28mm]{geometry}
\usepackage{amsmath,amssymb}
\usepackage{mathrsfs}
\usepackage{tikz-cd}
\usepackage{microtype}
\usepackage{xcolor}
\usepackage{hyperref}
\numberwithin{equation}{section}

\theoremstyle{plain}
\newtheorem{theorem}{Theorem}[section]
\newtheorem{proposition}[theorem]{Proposition}
\newtheorem{lemma}[theorem]{Lemma}
\newtheorem{corollary}[theorem]{Corollary}
\theoremstyle{definition}
\newtheorem{definition}[theorem]{Definition}
\theoremstyle{remark}
\newtheorem{remark}[theorem]{Remark}
\newtheorem{example}[theorem]{Example}

\definecolor{AMSDarkBlue}{RGB}{0,51,102}
\hypersetup{
  colorlinks=true,
  linkcolor=AMSDarkBlue,
  citecolor=AMSDarkBlue,
  urlcolor=AMSDarkBlue,
  breaklinks=true,
  bookmarksnumbered=true,
  bookmarksopen=true,
  bookmarksopenlevel=2,
  bookmarksdepth=3,
  pdfpagemode=UseOutlines,
  pdftitle={Indecomposable Rational Functions over Finite Fields with Galois Closure of Genus One: Reconstruction and Exceptionality},
  pdfauthor={Xiang Fan}
}

\begin{document}

\title[Genus-one reconstruction and exceptionality]{Indecomposable Rational Functions over Finite Fields with Galois Closure of Genus One: Reconstruction and Exceptionality}
\author{Xiang Fan}
\address{School of Mathematics, Sun Yat-sen University, Guangzhou 510275, China}
\email{fanx8@mail.sysu.edu.cn}

\begin{abstract}
Let $k=\mathbb F_q$.  We reconstruct every $k$-indecomposable
$g\in k(X)$ of degree greater than one whose normal closure has genus one.
Such a map is automatically separable and, after independent degree-one
changes of the source and target coordinates over $k$, arises from a separable
equivariant isogeny between elliptic curves equipped with compatible finite
group actions stable under Frobenius conjugation.  In particular,
$\deg g=\ell$ or $\ell^2$ for a prime $\ell$, with
$\ell\ne\operatorname{char}k$ in the latter case.  More generally, every
separable rational function with genus-one Galois closure admits a canonical
factorization class, modulo degree-one changes of the intermediate coordinates
over $k$, determined by the intrinsic translation subgroup of its geometric
monodromy group; every $k$-indecomposable factor of the remaining map has
Galois closure of genus zero.  For every equivariant-isogeny quotient and every
finite extension $k_r/k$, the same exact finite-kernel condition characterizes
both permutation of $\mathbf P^1(k_r)$ and exceptionality over $k_r$.  The
finite kernel and its induced Frobenius and linear symmetry actions also
determine the arithmetic and geometric monodromy permutation groups,
decomposition classes, and the exact periodic set of permutation extension
degrees, including its least period and limiting proportion.  Together with
the corresponding genus-zero theorem, these results yield permutation if and
only if exceptionality over every finite extension for every separable rational
function whose Galois closure has genus at most one; under $k$-decomposition,
the common extension-degree set is the intersection of the corresponding sets
for the factors.
\end{abstract}

\maketitle

\section{Introduction}

Let $k=\mathbb F_q$ and fix an algebraic closure $\bar k$.  Choose
$\mathbf x$ transcendental over $\bar k$, and throughout work inside one fixed
algebraic closure of $\bar k(\mathbf x)$.  Write $\operatorname{PGL}_2(k)$ for
the group of degree-one rational functions over $k$.  We call
$f,g\in k(X)$ \emph{$k$-M\"obius equivalent} if
\[
 g=\alpha\circ f\circ\beta
 \qquad\text{for some }\alpha,\beta\in\operatorname{PGL}_2(k).
\]
Thus the source and target coordinates may be changed independently.  We call
$f\in k(X)$ \emph{$k$-decomposable} if
$f=u\circ v$ for some $u,v\in k(X)$ with $\deg u,\deg v>1$; otherwise $f$ is
\emph{$k$-indecomposable}.  We call $f$ \emph{geometrically indecomposable} if
it is $\bar k$-indecomposable.

For every nonconstant $h\in k(X)$, let $\Omega_h$ denote the normal closure of
$k(\mathbf x)/k(h(\mathbf x))$ inside this fixed algebraic closure of
$\bar k(\mathbf x)$, put
\[
 \kappa_h:=\Omega_h\cap\bar k,
\]
and let $C_h/\kappa_h$ be the smooth projective curve with function field
$\Omega_h$.  By the \emph{genus of the normal closure of $h$} we mean the
genus of $C_h$.  Put
\[
 \overline{\Omega}_h:=\Omega_h\bar k.
\]
Extension of constants does not change the degree of the rational-function
extension, so $\overline{\Omega}_h$ is the normal closure of
$\bar k(\mathbf x)/\bar k(h(\mathbf x))$; moreover,
$(C_h)_{\bar k}$ has the same genus as $C_h$.  A nonconstant $h$ is
\emph{separable} if
$k(\mathbf x)/k(h(\mathbf x))$ is separable; equivalently,
$h'(X)\not\equiv0$.  For separable $h$, $\Omega_h$ and
$\overline{\Omega}_h$ are respectively its arithmetic and geometric Galois
closures.  Their Galois groups are the arithmetic and geometric monodromy
groups, in the standard cover terminology of \cite[\S1.1, p.~371]{Fri05}.

We use \emph{isogeny} for a nonzero morphism of elliptic curves sending origin
to origin, i.e. the nonzero case of the convention in
\cite[Chap.~III, \S4, p.~66]{Sil09}.  Our first main result reconstructs, in
arbitrary characteristic, every $k$-indecomposable map of degree greater than
one whose normal closure has genus one from the finite-field elliptic-isogeny
construction in Definition~\ref{def:equivariant-isogeny-datum}.  Our second
main result gives an exact finite-kernel criterion for permutation and
exceptionality over every finite extension.  The same finite data determine
monodromy, decomposition, and the full arithmetic of permutation extension
degrees.

The construction is summarized by the commutative diagram
\[
\begin{tikzcd}[column sep=large]
 E \arrow[r,"\varphi"] \arrow[d,"\pi"']
   & E' \arrow[d,"\pi'"] \\
 E/\Lambda \arrow[r,"f"']
   & E'/\Lambda .
\end{tikzcd}
\]
Every automorphism $\gamma$ of an elliptic curve $E$ admits a unique affine
decomposition
\[
 \gamma_E(P)=u_{\gamma,E}(P)+a_{\gamma,E},
 \qquad
 u_{\gamma,E}\in\operatorname{Aut}_{\bar k}(E,\mathrm O_E),
 \quad a_{\gamma,E}=\gamma_E(\mathrm O_E).
\]
The assignment $\gamma\mapsto u_{\gamma,E}$ is a homomorphism.

\begin{definition}[Equivariant isogeny datum]
\label{def:equivariant-isogeny-datum}
An \emph{equivariant isogeny datum} over $k$ consists of elliptic curves
$E,E'/k$, a finite abstract group $\Lambda$ with prescribed faithful actions
on $E_{\bar k}$ and $E'_{\bar k}$ by curve automorphisms, and a separable
$\Lambda$-equivariant $k$-isogeny $\varphi:E\to E'$, such that the prescribed
action on $E$ is invariant under conjugation by the $q$-power bijection of
$E(\bar k)$ and the image of $\gamma\mapsto u_{\gamma,E}$ is
nontrivial.
\end{definition}

Lemmas~\ref{lem:target-frobenius} and~\ref{lem:quotient} show that both
quotient curves in the diagram are defined over $k$ and are $k$-isomorphic to
$\mathbf P^1$.  Equivariance therefore determines a unique induced
$k$-morphism $f:E/\Lambda\to E'/\Lambda$ satisfying
$\pi'\circ\varphi=f\circ\pi$, with $\deg f=\deg\varphi$.  Its
$k$-M\"obius equivalence class is intrinsic to the datum.  For $m\geqslant1$,
let $C_m$ denote the cyclic group of order $m$.

The construction is exhaustive for degree-greater-than-one $k$-indecomposable
maps whose normal closure has genus one.

\begin{theorem}[Finite-field reconstruction from a genus-one normal closure]
\label{thm:intro-reconstruction}
Let $g\in k(X)$ have degree greater than one and be $k$-indecomposable.
Suppose that the normal closure of $g$ has genus one.  Then $g$ is
automatically separable and is $k$-M\"obius equivalent to the induced quotient
map associated with an equivariant isogeny datum over $k$.

For every equivariant-isogeny datum obtained by this reconstruction, put
\(\mathcal N:=\ker\varphi(\bar k).\)  Then, for a prime $\ell$, the isogeny kernel and degree satisfy
\[
 \mathcal N\cong C_\ell^e,
 \qquad e\in\{1,2\},
 \qquad \deg g=\ell^e,
 \qquad e=2\Longrightarrow \ell\ne\operatorname{char}k.
\]
The induced symmetry action on $\mathcal N$ is faithful, and the subgroup of
$\operatorname{Aut}(\mathcal N)$ generated by this symmetry action and the
action induced by $q$-power Frobenius acts \emph{irreducibly}: there is no
nonzero proper invariant $\mathbb F_\ell$-subspace of $\mathcal N$.
\end{theorem}

No exceptionality hypothesis is used in Theorem~\ref{thm:intro-reconstruction},
and no prime-to-characteristic ramification hypothesis is imposed.  More
generally, Theorem~\ref{thm:intrinsic-translation-factorization} associates to
every separable map whose Galois closure has genus one a factorization class in
which factorizations differing by degree-one changes of the intermediate
coordinates over $k$ are identified.  This class is determined by the intrinsic
translation subgroup of its geometric monodromy group; every
$k$-indecomposable factor of the remaining map has Galois closure of genus
zero.  In characteristic greater than $3$, the remaining map itself has Galois
closure of genus zero.

For a finite field $K$, write $\mathbf P^1(K)=K\cup\{\infty\}$.  A rational
function $h\in K(X)$ is a \emph{permutation rational function over $K$} if it
induces a bijection of $\mathbf P^1(K)$, and it is \emph{exceptional over $K$}
if its base change is a permutation over infinitely many finite extensions of
$K$.  The two notions need not agree for arbitrary rational functions over the
ground field: there are nonexceptional permutation rational functions of
degree $4$ over small finite fields \cite[Theorem~1.4]{DZ22}.

\begin{definition}[Stability under Frobenius conjugation]
A prescribed faithful $\Lambda$-action on $E_{\bar k}$ by curve automorphisms
is \emph{stable under $q$-power Frobenius conjugation} if its image in
$\operatorname{Aut}_{\bar k}(E)$ is invariant under conjugation by the
$q$-power bijection of $E(\bar k)$.
\end{definition}

For an equivariant isogeny datum, write $\mathcal F_E:E\to E$ for the
$q$-power Frobenius endomorphism and also for its induced bijection on
$E(\bar k)$, and put
\[
 \mathcal N:=\ker\varphi(\bar k),
 \qquad k_r:=\mathbb F_{q^r},
 \qquad L_E:=\{u_{\gamma,E}:\gamma\in\Lambda\}.
\]
For a finite set $S$, write $\lvert S\rvert$ for its cardinality.

\begin{theorem}[Kernel criterion and permutation--exceptionality equivalence]
\label{thm:intro-kernel}
Let $f$ be associated with an equivariant isogeny datum over $k$.  For every
$r\geqslant1$, the following are equivalent:
\begin{enumerate}
\item[\textup{(i)}] $f$ permutes $\mathbf P^1(k_r)$;
\item[\textup{(ii)}] $\mathcal N\cap\ker(\mathcal F_E^r-u_{\gamma,E})=\{\mathrm O_E\}$
for every $\gamma\in\Lambda$ (equivalently, for every $u\in L_E$);
\item[\textup{(iii)}] $f$ is exceptional over $k_r$.
\end{enumerate}
\end{theorem}

\begin{corollary}[Indecomposable genus-one permutation--exceptionality equivalence]
\label{cor:intro-exhaustive}
Let $g\in k(X)$ have degree greater than one, be $k$-indecomposable,
and have Galois closure of genus one.  Then for every $r\geqslant1$,
\[
 g\text{ permutes }\mathbf P^1(k_r)
 \Longleftrightarrow
 g\text{ is exceptional over }k_r.
\]
For any datum reconstructed in Theorem~\ref{thm:intro-reconstruction}, this
is equivalent to
\[
 \mathcal N\cap
 \ker(\mathcal F_E^r-u_{\gamma,E})=\{\mathrm O_E\}
 \qquad\text{for every }\gamma\in\Lambda.
\]
\end{corollary}

For any rational function $h\in k(X)$, define
\[
 \mathcal D_{\mathrm{perm}}(h)
 :=\{r\geqslant1:h\text{ permutes }\mathbf P^1(k_r)\},
 \qquad
 \mathcal D_{\mathrm{exc}}(h)
 :=\{r\geqslant1:h\text{ is exceptional over }k_r\}.
\]
We call their elements \emph{permutation extension degrees} and
\emph{exceptional extension degrees}, respectively.  Here $h$ is exceptional
over $k_r$ exactly when its base change to $k_r$ permutes
$\mathbf P^1(k_{rs})$ for infinitely many $s\geqslant1$; equivalently,
\[
 r\in\mathcal D_{\mathrm{exc}}(h)
 \Longleftrightarrow
 \{s\geqslant1:rs\in\mathcal D_{\mathrm{perm}}(h)\}\text{ is infinite}.
\]
When these two sets coincide, we write \(\mathcal D(h)\) for their common set.

For a field $K$, a \emph{$K$-decomposition} of a rational function
$f\in K(X)$ is an expression
\[
 f=f_s\circ\cdots\circ f_1,
 \qquad f_i\in K(X),\quad \deg f_i>1.
\]
Two $K$-decompositions $(f_1,\ldots,f_s)$ and
$(f_1',\ldots,f_s')$ are \emph{equivalent} if there exist
$\mu_0,\ldots,\mu_s\in\operatorname{PGL}_2(K)$ with $\mu_0=\mu_s=1$ and
\[
 f_i'=\mu_i\circ f_i\circ\mu_{i-1}^{-1}
 \qquad\text{for }1\leqslant i\leqslant s.
\]
A $K$-decomposition is \emph{complete} if every factor is
$K$-indecomposable.

\begin{theorem}[Permutation--exceptionality equivalence for genus at most one]
\label{thm:intro-genus-at-most-one-permutation-exceptionality}
Let $f\in k(X)$ be separable of degree greater than one, and suppose that the
Galois closure of $f$ has genus at most one.  Then for every
$r\geqslant1$,
\[
 f\text{ permutes }\mathbf P^1(k_r)
 \Longleftrightarrow
 f\text{ is exceptional over }k_r.
\]
No indecomposability hypothesis is required.  Equivalently,
$\mathcal D_{\mathrm{perm}}(f)=\mathcal D_{\mathrm{exc}}(f)$.
\end{theorem}

\begin{theorem}[Extension degrees under functional decomposition]
\label{thm:intro-genus-at-most-one-decomposition-extension-degrees}
Let $f\in k(X)$ be separable of degree greater than one, and suppose that the
Galois closure of $f$ has genus at most one.  For any $k$-decomposition,
\[
 f=f_s\circ\cdots\circ f_1,
 \qquad \deg f_i>1,
\]
every $f_i$ is separable and has Galois closure of genus at most one.
Thus \(\mathcal D(f)\) and every \(\mathcal D(f_i)\) are defined, and
\[
 \mathcal D(f)=\bigcap_{i=1}^s\mathcal D(f_i).
\]
In particular, for any complete decomposition the common set of permutation
and exceptional extension degrees is obtained by intersecting the corresponding
sets for its $k$-indecomposable genus-zero and genus-one factors; the resulting set
is independent of the chosen complete decomposition.
\end{theorem}

Thus the two principal genus-one results are
Theorems~\ref{thm:intro-reconstruction} and~\ref{thm:intro-kernel}: the first
reconstructs the $k$-indecomposable genus-one class, while the second solves its
permutation--exceptionality problem over every finite extension by a finite
kernel test.  Combining this theory with the independently proved genus-zero
theorem \cite{FanG0} gives
Theorems~\ref{thm:intro-genus-at-most-one-permutation-exceptionality}
and~\ref{thm:intro-genus-at-most-one-decomposition-extension-degrees}.  The
key additional observation is that the Galois-closure curve of every factor is
realized by a quotient of the original Galois-closure curve and hence again has
genus zero or one.  Proofs are given in Section~\ref{sec:extensions}.

For multiplication maps $\varphi=[n]$ with $\gcd(n,q)=1$, condition
\textup{(ii)} in Theorem~\ref{thm:intro-kernel} becomes a finite determinant
test on $E[n]$.  More generally, the kernel together with its Frobenius and
induced linear symmetry actions determines affine monodromy, decomposition
classes, and all permutation extension degrees.

The geometric and Galois-theoretic antecedents are classical.  General
exceptional-cover theory relates exceptionality to fixed points of Frobenius
cosets in the arithmetic monodromy group; see Fried
\cite[Corollary~2.8]{Fri05} and Guralnick--Tucker--Zieve
\cite[Lemmas~4.1 and~4.3]{GTZ07}.  For the genus-one geometry, a
characteristic-zero precursor is Guralnick--M\"uller--Saxl: their
Theorems~6.5 and~6.6 give characteristic-zero elliptic quotient descriptions
for indecomposable Euclidean rational functions, including shifted quotient
automorphisms, while Lemma~6.10 gives the corresponding affine monodromy
calculation \cite[Theorems~6.5--6.6 and Lemma~6.10]{GMS03}.  Pakovich gives a
broader geometric description of complex rational functions whose
normalization has genus one \cite[Theorem~1.2]{Pak18}.  More recently, Hyde
proves a characteristic-zero canonical cyclic-quotient factorization for maps
from a genus-one curve with genus-one Galois closure and gives the corresponding
characterization of generalized Latt\`es maps
\cite[Lemmas~3.2 and~3.4]{Hyde26}.  Thus the elliptic-quotient geometry itself
is not new here.

Kre\v{s}o--Zieve call a cover a \emph{coordinate projection} when there are
nonconstant vertical morphisms making it a quotient of a morphism of
one-dimensional algebraic groups \cite[pp.~3--4]{KZ14}.  Geometrically, the
maps considered here lie in this broader class, which includes coordinate
projections of elliptic isogenies.  When the isogeny is
an endomorphism, they call the resulting coordinate projection a Latt\`es map
\cite[pp.~3--4]{KZ14}.  Our equivariant-isogeny data form a more structured
finite-field subclass of this setting: the two vertical maps are quotients by
compatible prescribed finite affine actions of the same abstract group.

The standard origin-preserving $C_2$ Latt\`es subfamily over finite fields is
already well understood.  K\"u\c{c}\"uksakall{\i} determines value sets for
reductions of complex-multiplication (CM) elliptic endomorphisms and proves a
necessary and sufficient bijectivity criterion over finite extensions
\cite[Corollary~2.8]{Kuc14}.  For multiplication maps on arbitrary elliptic
curves over finite fields, Bell et al. obtain the equivalent criterion
\cite[Corollary~2.6]{Bell22} and note its essential equivalence with
K\"u\c{c}\"uksakall{\i}'s formulation \cite[Remark~2.7]{Bell22}.
Bisson--Tibouchi prove the corresponding kernel criterion in a prime-degree
$C_2$ setting and use it algorithmically to construct permutation rational
functions \cite[Theorem~1 and Section~4]{BT18}.  Bell et al. primarily study
the density of periodic points inside $\mathbf P^1(\mathbb F_{q^r})$, which is
different from the limiting proportion considered here of extension degrees
$r$ for which the induced rational function is a permutation.
Panraksa--Samart--Sriwongsa study a different arithmetic notion of
exceptionality for Latt\`es maps \cite{PSS26}.

Against this background, our new contributions are the converse reconstruction
over a prescribed finite field, in arbitrary characteristic, for the full
class of degree-greater-than-one $k$-indecomposable rational functions whose
normal closure has genus one, without an exceptionality or tameness
hypothesis, and the resulting theory over all finite extensions for general
equivariant-isogeny quotients.  The finite kernel and its induced Frobenius and
linear symmetry actions give an exact permutation--exceptionality criterion
over every finite extension and determine the monodromy permutation groups,
decomposition classes, and extension-degree arithmetic.  Combining this
genus-one theory with the independent genus-zero theorem then gives the
genus-at-most-one consequences.

A separate companion manuscript \cite{FanEERF} treats the tame
$k$-indecomposable exceptional locus and develops the finer prescribed-field
classification there: affine translation parameters and branch arithmetic,
exact $k$-M\"obius equivalence, prescribed-field models, occurrence, and exact
enumeration.  The general reconstruction, factorization, monodromy, base-change,
and extension-degree results proved here apply to that locus as special cases;
none of the companion's fixed-field classification or counting results is used
here.

A key ingredient is the family of sets defined by
$\mathcal F_E(P)=\gamma_E(P)$
\[
 \mathcal S_\gamma(E)
 :=\{P\in E(\bar k):\mathcal F_E(P)=\gamma_E(P)\}.
\]
Unlike in the origin-preserving special case, such a set need not be a subgroup:
$\ker(\mathcal F_E-u_{\gamma,E})$ acts freely and transitively on it by
translations.  Section~\ref{sec:frobenius-equations-elliptic-quotients}
uses the exact identity
\[
 \sum_{\gamma\in\Lambda}
 \lvert\mathcal S_\gamma(E)\cap\pi^{-1}(Q)\rvert
 =\lvert\Lambda\rvert
 \qquad\text{for every }Q\in(E/\Lambda)(k),
\]
which exactly compensates for stabilizer multiplicities above branch points.
Section~\ref{sec:isogenies} uses these fixed-point sets to prove
Theorem~\ref{thm:intro-kernel}; the additional issue for affine actions is to
show that a nonzero kernel point cannot be hidden inside an affine quotient
orbit.

To express the kernel criterion in finite terms, we record the induced
Frobenius and symmetry actions on the kernel, together with the associated
monodromy and constant-field notation.  Define
\[
 \begin{aligned}
 \Lambda_{\mathcal N}
 &:=\operatorname{im}\bigl(
 \Lambda\longrightarrow\operatorname{Aut}(\mathcal N),\
 \gamma\longmapsto u_{\gamma,E}|_{\mathcal N}\bigr),
 &\qquad \mathcal F_{\mathcal N}&:=\mathcal F_E|_{\mathcal N},\\
 \mathcal A_{\mathcal N}
 &:=\langle\Lambda_{\mathcal N},\mathcal F_{\mathcal N}\rangle,
 &\qquad \mathcal C_{\mathcal N}&:=
 \mathcal A_{\mathcal N}/\Lambda_{\mathcal N},\\
 \overline{\mathcal F}_{\mathcal N}
 &:=\mathcal F_{\mathcal N}\Lambda_{\mathcal N},
 &\qquad c_{\mathcal N}&:=
 \operatorname{ord}_{\mathcal C_{\mathcal N}}
 (\overline{\mathcal F}_{\mathcal N}).
 \end{aligned}
\]
Let $G_f$ and $A_f$ denote the geometric and arithmetic monodromy groups of
$f$, respectively.  Theorem~\ref{thm:affine-monodromy} gives
\[
 G_f\cong\mathcal N\rtimes\Lambda_{\mathcal N},\qquad
 A_f\cong\mathcal N\rtimes\mathcal A_{\mathcal N},\qquad
 A_f/G_f\cong\mathcal C_{\mathcal N}.
\]
For $r\geqslant1$, let $A_f^{(r)}$ denote the arithmetic monodromy group
after base change to $k_r$.  Corollary~\ref{cor:base-change-monodromy}
gives the corresponding arithmetic monodromy after every constant extension:
over $k_r$,
\[
 A_f^{(r)}\cong
 \mathcal N\rtimes
 \langle\Lambda_{\mathcal N},\mathcal F_{\mathcal N}^{\,r}\rangle,
 \qquad
 [A_f^{(r)}:G_f]
 =\frac{c_{\mathcal N}}{\gcd(c_{\mathcal N},r)},
\]
and the algebraic closure of $k_r$ in the arithmetic Galois closure is
$\mathbb F_{q^{\operatorname{lcm}(c_{\mathcal N},r)}}$; equivalently, this is
the full constant field of the arithmetic Galois closure over $k_r$.
The affine realization also yields the decomposition classification.
Proposition~\ref{prop:decomposition}, obtained from \cite[Lemma~2.8]{KZ14},
identifies decompositions over $k$ with $\mathcal A_{\mathcal N}$-stable
subgroup chains of $\mathcal N$.  Corollary~\ref{cor:prime-prime-square-factor-degrees}
shows that the factors in a complete decomposition have only prime or
prime-square degree.  Let
$\mathcal C_{\mathcal N}^{\mathrm{fpf}}\subseteq\mathcal C_{\mathcal N}$ be
the set of cosets none of whose elements fixes a point of
$\mathcal N\setminus\{\mathrm O_E\}$.

As a quantitative refinement of Theorem~\ref{thm:intro-kernel}, the finite
module also yields the complete arithmetic of permutation extension degrees.

\begin{theorem}[Criterion for permutation extension degrees]
\label{thm:intro-permutation-extension-criterion}
For the quotient map $f$ associated with an equivariant isogeny datum over
$k$, the set of extension degrees on which $f$ is a permutation is
\[
 \mathcal D_{\mathrm{perm}}(f)
 =\{r\geqslant1:
 \overline{\mathcal F}_{\mathcal N}^{\,r}
 \in\mathcal C_{\mathcal N}^{\mathrm{fpf}}\}.
\]
Consequently this set is purely periodic modulo $c_{\mathcal N}$.  Write
$\operatorname{per}(f)$ for its least positive period and
\[
 \delta_f:=\lim_{M\to\infty}\frac{1}{M}
 \bigl\lvert\mathcal D_{\mathrm{perm}}(f)\cap\{1,\ldots,M\}\bigr\rvert
\]
for its limiting proportion among positive integers; the limit exists by
periodicity.  The empty set has least positive period $1$ and limiting
proportion $0$.  Define the setwise stabilizer for the left-translation action by
\[
 H_{\mathcal N}
 :=\operatorname{Stab}_{\mathcal C_{\mathcal N}}
 (\mathcal C_{\mathcal N}^{\mathrm{fpf}})
 =\{\xi\in\mathcal C_{\mathcal N}:
 \xi\mathcal C_{\mathcal N}^{\mathrm{fpf}}
 =\mathcal C_{\mathcal N}^{\mathrm{fpf}}\}.
\]
Then
\[
 \operatorname{per}(f)=[\mathcal C_{\mathcal N}:H_{\mathcal N}],
 \qquad
 \delta_f=
 \frac{\lvert\mathcal C_{\mathcal N}^{\mathrm{fpf}}\rvert}{c_{\mathcal N}}.
\]
Moreover, $f$ is exceptional over $k$ if and only if
$\mathcal C_{\mathcal N}^{\mathrm{fpf}}\ne\varnothing$.  In the complementary
case $\mathcal D_{\mathrm{perm}}(f)=\varnothing$, so
$\operatorname{per}(f)=1$ and $\delta_f=0$, in agreement with the formulas
above.  If
$\mathcal N\ne\{\mathrm O_E\}$, then the identity coset does not belong to
$\mathcal C_{\mathcal N}^{\mathrm{fpf}}$, so $c_{\mathcal N}\mid r$ implies
$r\notin\mathcal D_{\mathrm{perm}}(f)$.
\end{theorem}

Proposition~\ref{prop:tameness} gives
\[
 e_{\pi(P)}(f)
 =[
 \operatorname{Stab}_\Lambda(\varphi(P)):
 \operatorname{Stab}_\Lambda(P)]
 \mid \lvert L_E\rvert.
\]
Thus, even if $\Lambda$ contains translations of characteristic order, every
ramification index is prime to $\operatorname{char}k$ when
$\operatorname{char}k>3$.
The reconstruction theorem itself is valid in every characteristic.  In characteristics $2$ and $3$, separability is automatic in the
$k$-indecomposable genus-one setting, but ramification indices need not be prime to
the characteristic.  When $\operatorname{char}k>3$, one has $L_E\cong C_m$ for
some $m\in\{2,3,4,6\}$.  Writing $L_E=\langle\beta\rangle$, Frobenius conjugation sends $\beta$ to
$\beta^q$.  This gives the following base-field criteria:
\begin{equation*}
\begin{array}{c|c|c}
L_E&\text{Frobenius action on }L_E&\text{permutation criterion}\\ \hline
C_2&\text{trivial}&
\mathcal N\cap E(\mathbb F_{q^2})=\{\mathrm O_E\}\\
C_3,\ q\equiv1\pmod3&\text{trivial}&
\mathcal N\cap E(\mathbb F_{q^3})=\{\mathrm O_E\}\\
C_3,\ q\equiv2\pmod3&\text{inversion}&
\mathcal N\cap E(k)=\{\mathrm O_E\}\\
C_4,\ q\equiv1\pmod4&\text{trivial}&
\mathcal N\cap E(\mathbb F_{q^4})=\{\mathrm O_E\}\\
C_4,\ q\equiv3\pmod4&\text{inversion}&
\mathcal N\cap E(\mathbb F_{q^2})=\{\mathrm O_E\}\\
C_6,\ q\equiv1\pmod6&\text{trivial}&
\mathcal N\cap E(\mathbb F_{q^6})=\{\mathrm O_E\}\\
C_6,\ q\equiv5\pmod6&\text{inversion}&
\mathcal N\cap E(\mathbb F_{q^2})=\{\mathrm O_E\}.
\end{array}
\end{equation*}
The cyclic quotient geometries in this table belong to the classical theory
of elliptic-isogeny coordinate projections and Latt\`es maps.  The table places
them under a single finite-field permutation criterion derived from the same
kernel condition.  The $C_2$ row is the classical finite-field
Latt\`es criterion discussed above; the other rows concern different cyclic
linear symmetries under the same finite-kernel criterion.

For rank-two kernels with cyclic linear symmetry, these base-field criteria
extend to closed formulas for all permutation extension degrees.  Under the
hypotheses of Section~\ref{subsec:explicit-rank-two-formulas},
Theorem~\ref{thm:split-spectral-criterion} treats the case in which Frobenius acts
trivially on the cyclic symmetry group, expressing the extension-degree set
through the orders of the two Frobenius eigenvalues, equivalently through the
divisibility of $\lvert E(\mathbb F_{q^{mr}})\rvert$ by $\ell$.
Theorem~\ref{thm:nonsplit-parity-residue-criterion} treats Frobenius acting by inversion
and gives an explicit odd/even criterion; Corollary~\ref{cor:nonsplit-constant-field-period}
gives the constant-field degree, least period, and limiting proportion.  These hypotheses include the cyclic cases described above whose ramification
indices remain prime to the characteristic in characteristics $2$ and $3$.

In particular, the condition
$\mathcal N\cap E(\mathbb F_{q^2})=\{\mathrm O_E\}$ reappears in the
nonsplit $C_4$ and $C_6$ rows because Frobenius acts by inversion on the cyclic
symmetry group; this does not identify the resulting quotient maps with the
standard $C_2$ Latt\`es family.  Section~\ref{sec:cyclic} proves these formulas and
their point-count versions.  Section~\ref{sec:examples} gives examples for
the different cyclic symmetries and a non-origin-preserving $C_2$-action over
$\mathbb F_5$ whose action is not $\mathbb F_5$-conjugate to an
origin-preserving involution.

The paper is organized as follows.  Section~\ref{sec:reconstruction} proves
the finite-field reconstruction and the canonical factorization determined by
the translation subgroup.  Section~\ref{sec:frobenius-equations-elliptic-quotients} develops the twisted
Frobenius fixed-point sets and the weighted counting identity, and
Section~\ref{sec:isogenies} converts these results into the exact finite-kernel
criterion.  Section~\ref{sec:extensions} develops affine monodromy, decomposition
classes, and the full set of permutation extension degrees.
Section~\ref{sec:cyclic} derives the closed cyclic formulas, and
Section~\ref{sec:examples} gives explicit realizations separating the effects
of linear quotient symmetry and affine translation components.

\section{Genus-one reconstruction over finite fields and the intrinsic translation subgroup}
\label{sec:reconstruction}

This section proves Theorem~\ref{thm:intro-reconstruction} and the more
general canonical factorization determined by the intrinsic translation
subgroup.  We first exclude the $k$-indecomposable inseparable case.  We then
use intermediate fields descending to $k$ and the intrinsic translation
subgroup of the genus-one Galois closure to recover the equivariant-isogeny
datum over the ground field.

\subsection{Intermediate fields over the ground field}

\begin{lemma}[The inseparable indecomposable case]
\label{lem:inseparable-indecomposable}
Let $k=\mathbb F_q$ have characteristic $p$, and let $g\in k(X)$ have
$\deg g>1$.  If $g$ is inseparable and $k$-indecomposable, then
\[
 g=\mu\circ X^p
\]
for some $\mu\in\operatorname{PGL}_2(k)$.  In particular, the inseparable
$k$-indecomposable case is disjoint from the genus-one normal-closure setting
considered below.
\end{lemma}

\begin{proof}
Since $g$ is inseparable, $g'(X)=0$.  The kernel of the derivation on
$k(X)$ is $k(X^p)$, so $g(X)=g_0(X^p)$ for some $g_0\in k(X)$.  Thus
$g=g_0\circ X^p$.  Since $g$ is $k$-indecomposable and $\deg X^p=p>1$, one
has $\deg g_0=1$, proving the first assertion.  Hence
$k(\mathbf x)/k(g(\mathbf x))=k(\mathbf x)/k(\mathbf x^p)$ is purely
inseparable and already normal.  Its normal closure therefore has function
field $k(\mathbf x)$ and genus zero, proving the final assertion.
\end{proof}

Now let $g\in k(X)$ be separable of degree greater than one, and put
$\mathbf t:=g(\mathbf x)$.  By the convention fixed in the Introduction,
$\overline{\Omega}_g/\bar k(\mathbf t)$ is the geometric Galois closure of
$\bar k(\mathbf x)/\bar k(\mathbf t)$.  Put
\[
 C:=(C_g)_{\bar k},
 \qquad
 G:=\operatorname{Gal}(\overline{\Omega}_g/\bar k(\mathbf t)),
 \qquad
 H:=\operatorname{Gal}(\overline{\Omega}_g/\bar k(\mathbf x)),
\]
and let
\[
 \widetilde G:=\operatorname{Aut}_{k(\mathbf t)}(\overline{\Omega}_g),
 \qquad
 \widetilde H:=\operatorname{Aut}_{k(\mathbf x)}(\overline{\Omega}_g).
\]
The group $\Gamma_k:=\operatorname{Gal}(\bar k/k)$ acts canonically
coefficientwise on $\bar k(\mathbf x)$, fixing $\mathbf x$ and $\mathbf t$.
An extension of an element of $\Gamma_k$ to the fixed algebraic
closure need not be unique; however, every such extension preserves
$\overline{\Omega}_g$ by the intrinsic minimality of the normal closure.
Thus $\widetilde G$ and $\widetilde H$ are the resulting profinite
semilinear descent groups, distinct from the finite geometric monodromy group
$G$.  One has $G\triangleleft\widetilde G$, $H=G\cap\widetilde H$, and
restriction to $\bar k$ gives the exact sequences
\[
 1\longrightarrow G\longrightarrow\widetilde G\longrightarrow
 \Gamma_k\longrightarrow1,
 \qquad
 1\longrightarrow H\longrightarrow\widetilde H\longrightarrow
 \Gamma_k\longrightarrow1.
\]
Because $\overline{\Omega}_g$ is the Galois closure of
$\bar k(\mathbf x)/\bar k(\mathbf t)$, the action of $G$ on $G/H$ is
faithful; equivalently,
\[
 \operatorname{core}_G(H)
 =\bigcap_{a\in G}aHa^{-1}
 =1.
\]

\begin{proposition}[Descent criterion for intermediate fields]
\label{prop:arithmetic-intermediate}
Let $H\leq J\leq G$.  The field $\overline{\Omega}_g^J$ is the base change of a
$k$-intermediate field
\[
 k(\mathbf t)\subseteq L_J\subseteq k(\mathbf x)
\]
if and only if $J$ is normalized by $\widetilde H$.  In this case
\[
 \begin{aligned}
 L_J&=\overline{\Omega}_g^{J\widetilde H},
 &\qquad L_J\bar k&=\overline{\Omega}_g^J,\\
 [k(\mathbf x):L_J]&=[J:H],
 &\qquad [L_J:k(\mathbf t)]&=[G:J].
 \end{aligned}
\]
Consequently, $g$ is $k$-indecomposable if and only if there is no
subgroup $J$ normalized by $\widetilde H$ with $H<J<G$.
\end{proposition}

\begin{proof}
If $J$ is normalized by $\widetilde H$, then $J\widetilde H$ is a subgroup
of $\widetilde G$ and
\((J\widetilde H)\cap G=J,\)
because $\widetilde H\cap G=H\leq J$.  Hence
$L_J=\overline{\Omega}_g^{J\widetilde H}$ lies between $k(\mathbf t)$ and $k(\mathbf x)$, and
base change to $\bar k$ gives $L_J\bar k=\overline{\Omega}_g^J$.  The degree
formulas follow from Galois correspondence.

Conversely, if $L_J\subseteq k(\mathbf x)$ is a $k$-intermediate field with
$L_J\bar k=\overline{\Omega}_g^J$, then every element of $\widetilde H$ fixes
$k(\mathbf x)$ and preserves $L_J\bar k$, so it normalizes $J$.  The final
assertion follows from L\"uroth's theorem and the correspondence between
functional decompositions and intermediate rational function fields.
\end{proof}

\subsection{The factorization from the intrinsic translation subgroup}

Assume from now on that the fixed geometric Galois-closure curve
$C=(C_g)_{\bar k}$ has genus one.  Its translation subgroup
$\operatorname{Trans}(C)$ is intrinsic; after choosing any point of $C$ as
origin it becomes the usual group of translations of the resulting elliptic
curve.  Define
\[
 \mathscr T:=G\cap\operatorname{Trans}(C),
 \qquad
 S:=\mathscr T\cap H,
 \qquad
 J:=\mathscr T H.
\]
The subgroup $\operatorname{Trans}(C)$ is characteristic in
$\operatorname{Aut}_{\bar k}(C)$, so $\mathscr T$ and $J$ are normalized by
$\widetilde H$.

\begin{theorem}[Factorization from the intrinsic translation subgroup]
\label{thm:intrinsic-translation-factorization}
With the notation above, there is a canonical $k$-intermediate field
$L_{\mathrm{tr}}\subseteq k(\mathbf x)$ characterized geometrically by
\[
 L_{\mathrm{tr}}\bar k
 =\bar k(\mathbf x)\cap\overline{\Omega}_g^{\mathscr T}
 =\overline{\Omega}_g^{\mathscr T H}.
\]
It determines, uniquely up to $k$-decomposition equivalence, a factorization
\[
 g=h\circ g_{\mathrm{tr}}
\]
with
\[
 \deg g_{\mathrm{tr}}=[\mathscr T:S],
 \qquad
 \deg h=[G:\mathscr T H].
\]
Moreover, $g_{\mathrm{tr}}$ is $k$-M\"obius equivalent to the
induced quotient map of an equivariant isogeny datum over $k$.
\end{theorem}

\begin{proof}
Apply Proposition~\ref{prop:arithmetic-intermediate} to
$J=\mathscr T H$.  Since $\overline{\Omega}_g^H=\bar k(\mathbf x)$, Galois correspondence
gives
\(\overline{\Omega}_g^{\mathscr T H} =\overline{\Omega}_g^{\mathscr T}\cap\overline{\Omega}_g^H,\)
which proves the intrinsic description of $L_{\mathrm{tr}}$.  L\"uroth's
theorem gives the factorization and the displayed degree formulas.

It remains to realize its right factor by an equivariant isogeny datum.
Recall that $\Gamma_k=\operatorname{Gal}(\bar k/k)\cong\widehat{\mathbb Z}$ as profinite groups,
and let $\operatorname{Fr}_q\in\Gamma_k$ denote the $q$-power automorphism.
The restriction map $\widetilde H\to\Gamma_k$ is surjective.  Since
$\widetilde H$ is a closed subgroup of the profinite group $\widetilde G$, it
is itself profinite.  Choose
$\widetilde{\operatorname{Fr}}_q\in\widetilde H$ above
$\operatorname{Fr}_q$.  The homomorphism
$\mathbb Z\to\widetilde H$, $n\mapsto\widetilde{\operatorname{Fr}}_q^{\,n}$,
is continuous for the profinite topology on $\mathbb Z$, and hence extends
uniquely to a continuous homomorphism
$s:\widehat{\mathbb Z}\to\widetilde H$.  Its composite with
$\widetilde H\to\Gamma_k$ is the identity on the dense subgroup generated by
$\operatorname{Fr}_q$, hence is the identity on $\Gamma_k$.  Thus
$\Sigma:=s(\Gamma_k)$ is a complement to $H$ in $\widetilde H$, so
\[
 \widetilde H=H\rtimes\Sigma.
\]

Since both $S$ and $\mathscr T$ are normalized by $\widetilde H$, the
products $S\Sigma$ and $\mathscr T\Sigma$ are subgroups.  Put
\[
 K_0:=\overline{\Omega}_g^{S\Sigma},
 \qquad
 K_1:=\overline{\Omega}_g^{\mathscr T\Sigma}.
\]
Because
\((S\Sigma)\cap G=S
\quad\text{and}\quad
(\mathscr T\Sigma)\cap G=\mathscr T,\)
base change to $\bar k$ gives
\[
 K_0\bar k=\overline{\Omega}_g^S,
 \qquad
 K_1\bar k=\overline{\Omega}_g^{\mathscr T}.
\]
Moreover $S\Sigma\leq\mathscr T\Sigma$, so $K_1\subseteq K_0$ and the
quotient morphism $C/S\to C/\mathscr T$ descends compatibly to $k$.

The geometric quotient of $C/S$ by $H/S$ has function field
$\overline{\Omega}_g^H=\bar k(\mathbf x)$, and its descended $k$-form has fixed field
\[
 \overline{\Omega}_g^{H\Sigma}
 =\overline{\Omega}_g^{\widetilde H}
 =k(\mathbf x).
\]
Likewise the geometric quotient of $C/\mathscr T$ by $H/S$ has function
field $\overline{\Omega}_g^{\mathscr T H}$, and its descended $k$-form has fixed field
\[
 \overline{\Omega}_g^{\mathscr T H\Sigma}
 =\overline{\Omega}_g^{\mathscr T\widetilde H}
 =L_{\mathrm{tr}}.
\]
Thus the quotient curves descend compatibly to genus-one curves over $k$,
and their quotients by $H/S$ are the two $k$-rational function fields above.
We continue to write $C/S$ and $C/\mathscr T$ for these descended $k$-forms.
Every genus-one curve over a finite field has a rational point by the
Hasse--Weil bound \cite[Theorem~5.2.3]{Sti09}.
Choose a point $O_0\in(C/S)(k)$ and let $O_1\in(C/\mathscr T)(k)$ be its
image under the descended quotient morphism.  Use $O_0$ and $O_1$ as origins
and write the resulting elliptic curves as $E_0$ and $E_1$.  The quotient
morphism now sends origin to origin and hence is a separable $k$-isogeny
\(\psi:E_0\longrightarrow E_1\)
whose geometric kernel is $\mathscr T/S$.  Put $\Lambda_0=H/S$.  Since $H$ normalizes both $S$ and $\mathscr T$, the same abstract group
$\Lambda_0$ acts on $E_0$ and $E_1$, and $\psi$ is equivariant.  Under the
compatible descent above, conjugation by
$\widetilde{\operatorname{Fr}}_q=s(\operatorname{Fr}_q)$ induces the
$q$-power Frobenius action on geometric points of the descended curves.
Hence both prescribed $\Lambda_0$-actions are stable under $q$-power
Frobenius conjugation.

Both actions are faithful.  Indeed, an element of $H/S$ acting trivially on
$C/S$ lies in $S$, and one acting trivially on
$C/\mathscr T$ lies in $H\cap\mathscr T=S$.  The image of the source linear-part homomorphism is nontrivial: otherwise
$H/S$ would act on $C/S$ only by translations, and
its quotient would again have genus one, whereas
\((C/S)/(H/S)=C/H\cong_{\bar k}\mathbf P^1.\)
Finally,
\(E_0/\Lambda_0\cong_k\mathbf P^1, \qquad E_1/\Lambda_0\cong_k\mathbf P^1,\)
because their function fields are respectively $k(\mathbf x)$ and
$L_{\mathrm{tr}}$.  The induced quotient map therefore represents
$g_{\mathrm{tr}}$ up to $k$-M\"obius equivalence.
\end{proof}

\begin{proposition}[Criterion for an equivariant-isogeny realization]
\label{prop:equivariant-isogeny-realization-criterion}
Let $g$ be separable with Galois closure of genus one.  Then $g$ is
$k$-M\"obius equivalent to the induced quotient map of an equivariant
isogeny datum if and only if
\[
 G=\mathscr T H,
\]
equivalently if and only if $\mathscr T$ is transitive on the sheets $G/H$.
In this case
\[
 \mathscr T\cap H=1,
 \qquad
 G=\mathscr T\rtimes H.
\]
\end{proposition}

\begin{proof}
The reverse implication follows from
Theorem~\ref{thm:intrinsic-translation-factorization}.

For the forward implication, choose an equivariant isogeny datum representing
$g$.  Over $\bar k$, the composite $E\to E'/\Lambda$ is Galois
with deck group $\mathcal H=\mathcal N\rtimes\Lambda$, where $\mathcal N$
acts by translations and $\Lambda$ by the prescribed affine automorphisms.
The lower cover corresponds to the subgroup $\Lambda\leq\mathcal H$.  Put
\[
 \mathcal K:=\operatorname{core}_{\mathcal H}(\Lambda).
\]
Then $E/\mathcal K$ realizes the geometric Galois-closure curve of the induced
quotient-map representative.  Since that representative is $k$-M\"obius
equivalent to $g$, this curve is $\bar k$-isomorphic to the fixed curve
$C=(C_g)_{\bar k}$.  Choose the induced $\bar k$-isomorphism and, for the
remainder of the proof, identify $C\cong E/\mathcal K$.
For the geometric monodromy group and point stabilizer used above,
\[
 G\cong\mathcal H/\mathcal K,
 \qquad
 H\cong\Lambda/\mathcal K.
\]
Under these identifications, the corresponding sheet sets $G/H$ and
$\mathcal H/\Lambda$ are naturally bijective.

For $\gamma\in\Lambda$ and $Q\in\mathcal N$,
\[
 \tau_Q\gamma\tau_Q^{-1}
 =\tau_{Q-u_{\gamma,E}(Q)}\gamma.
\]
It follows that
\[
 \mathcal K
 =\ker\!\left(
 \Lambda\longrightarrow\operatorname{Aut}(\mathcal N),\
 \gamma\longmapsto u_{\gamma,E}|_{\mathcal N}
 \right).
\]

By hypothesis $C$ has genus one.  Since $E$ also has genus one, the
Riemann--Hurwitz formula with different for $E\to C$ gives zero different.
Thus every nonidentity element of $\mathcal K$ acts without fixed points on
$E$.  An affine automorphism $P\mapsto u(P)+a$ with $u\ne1$ has a fixed point
because $1-u$ is a nonzero isogeny and hence
surjective.  Consequently every element of $\mathcal K$ is a translation.

Since $\mathcal K\leq\Lambda$ and $\mathcal N\cap\Lambda=1$, one has
$\mathcal N\cap\mathcal K=1$.  Moreover the description of $\mathcal K$
above shows that $\mathcal K$ centralizes $\mathcal N$.  Hence the
translations by $\mathcal N$ descend faithfully to genuine translations of
the genus-one curve $C$ under the chosen identification.  On the sheets of the lower cover they remain free and transitive: the map
\[
 \mathcal N\longrightarrow\mathcal H/\Lambda,
 \qquad Q\longmapsto\tau_Q\Lambda,
\]
is a bijection.
Thus the intrinsic translation subgroup
$\mathscr T=G\cap\operatorname{Trans}(C)$ is transitive on $G/H$, and
therefore $G=\mathscr T H$.

Finally, if $G=\mathscr T H$, then $\mathscr T\cap H$ is normalized by $H$
and centralized by $\mathscr T$, hence is normal in $G$.  Since
$\operatorname{core}_G(H)=1$, one obtains
\[
 \mathscr T\cap H=1,
 \qquad
 G=\mathscr T\rtimes H.
\]
\end{proof}

\begin{remark}[The factor generated by translations in the Galois closure]
The field $L_{\mathrm{tr}}$ is intrinsic at the level of intermediate fields:
\(L_{\mathrm{tr}}\bar k =\bar k(\mathbf x)\cap\overline{\Omega}_g^{\mathscr T}.\)
Thus $g_{\mathrm{tr}}$ is the maximal right factor generated by translations
already present in the genus-one Galois closure.  This does not assert maximality among all right factors that might admit an
elliptic-isogeny realization after further quotienting.
\end{remark}

\subsection{Self-normalization and the existence of translations}

\begin{lemma}[The subgroup $H$ is self-normalizing]
\label{lem:self-normalizing}
If $g$ is $k$-indecomposable, then
\[
 N_G(H)=H.
\]
\end{lemma}

\begin{proof}
The normalizer $N_G(H)$ is normalized by $\widetilde H$.  A strict inclusion
$H<N_G(H)<G$ would therefore contradict
Proposition~\ref{prop:arithmetic-intermediate}.  If $N_G(H)=G$, then
$H\triangleleft G$, so $\operatorname{core}_G(H)=1$ forces $H=1$.  This would
make $\overline{\Omega}_g=\bar k(\mathbf x)$ rational, contrary to the
genus-one hypothesis.
\end{proof}

\begin{lemma}[The special small-characteristic involution]
\label{lem:small-char-involution}
Suppose $p\in\{2,3\}$ and, after choosing an origin $O$, the group
$\operatorname{Aut}_{\bar k}(C,O)$ is noncyclic (the exceptional $j=0$ case).
Then it has a unique nontrivial involution.
\end{lemma}

\begin{proof}
In characteristic $3$, the special curve may be written
\(y^2=x^3-x,\)
and its twelve origin-preserving automorphisms are
\(\Phi_{u,r}(x,y)=(u^2x+r,u^3y), \qquad u^4=1, \quad r\in\mathbb F_3,\)
by \cite[Proposition~2.1]{KST17}.  Squaring shows that an involution must
have $u^2=1$ and $r=0$, leaving only the identity and the elliptic
involution.

In characteristic $2$, the special curve may be written
\(y^2+y=x^3,\)
and its twenty-four origin-preserving automorphisms are
\(\Phi_{u,r,s}(x,y) =\bigl(u^2x+r,\ y+u^2r^2x+s\bigr),\)
where $u\in\mathbb F_4^\times$, $r\in\mathbb F_4$, and
$s^2+s+r^3=0$ \cite[Proposition~3.1]{KST17}.  If
$\Phi_{u,r,s}^2=1$, comparison of the $x$-coordinate gives $u=1$, and then
comparison of the $y$-coordinate gives $r=0$.  Thus $s\in\{0,1\}$, again
leaving the identity and the elliptic involution.
\end{proof}

\begin{proposition}[Existence of a nontrivial translation subgroup]
\label{prop:nontrivial-translation-subgroup}
Let $g\in k(X)$ be separable, of degree greater than one, $k$-indecomposable,
and suppose that its Galois closure has genus one.  Then
\[
 \mathscr T\ne1.
\]
\end{proposition}

\begin{proof}
Assume $\mathscr T=1$.  After choosing an origin on $C$, the linear-part map
embeds $G$ into the finite group
$U=\operatorname{Aut}_{\bar k}(C,O)$.  Also $H\ne1$, since otherwise the
Galois-closure curve would be rational.

If $p>3$, then $U$ is cyclic of order $2$, $4$, or $6$; if
$p\in\{2,3\}$ and the automorphism group is not enlarged, then $U\cong C_2$.
In all these cases $G$ is abelian, so $H\triangleleft G$, contradicting
$\operatorname{core}_G(H)=1$ because $H\ne1$.

It remains to consider the enlarged groups in characteristics $2$ and $3$.
Let $z$ be their unique involution from
Lemma~\ref{lem:small-char-involution}; uniqueness makes $z$ central.  If
$\lvert H\rvert$ is even, Cauchy's theorem gives $z\in H$, so the nontrivial normal
subgroup $\langle z\rangle\leq H$ contradicts
$\operatorname{core}_G(H)=1$.  If $\lvert H\rvert$ is
odd, then $\lvert H\rvert=3$.  Since $[G:H]=\deg g>1$ and $\lvert G\rvert$ divides $12$ or $24$,
$\lvert G\rvert$ is even.  Hence $G$ contains an involution, necessarily $z$.  Since
$z\notin H$ is central,
\(H<N_G(H),\)
contradicting Lemma~\ref{lem:self-normalizing}.
\end{proof}

\subsection{Proof of the reconstruction theorem and the degree restriction}

\begin{proof}[Proof of Theorem~\ref{thm:intro-reconstruction}]
Lemma~\ref{lem:inseparable-indecomposable} excludes the inseparable case,
so $g$ is separable.  Proposition~\ref{prop:nontrivial-translation-subgroup} gives
$\mathscr T\ne1$.  Since $\mathscr T\triangleleft G$ and $H$ contains no nontrivial normal subgroup of $G$,
$\mathscr T$ cannot be contained in $H$; hence
\(H<\mathscr T H.\)
The subgroup $\mathscr T H$ is normalized by $\widetilde H$, so
indecomposability and Proposition~\ref{prop:arithmetic-intermediate} force
\(G=\mathscr T H.\)
Proposition~\ref{prop:equivariant-isogeny-realization-criterion} now reconstructs the required
equivariant isogeny datum, with
\(G=\mathscr T\rtimes H\)
and isogeny kernel identified with $\mathscr T$.

It remains to prove the structural assertions.  If
$1<C_0<\mathscr T$ is characteristic in $\mathscr T$, then $C_0H$ is
normalized by $\widetilde H$ and satisfies
\(H<C_0H<G,\)
contrary to indecomposability.  Thus the finite abelian group $\mathscr T$ has no nontrivial proper
characteristic subgroup.  If two distinct primes divided
$\lvert\mathscr T\rvert$, a Sylow subgroup would be a nontrivial proper
characteristic subgroup.  Hence $\mathscr T$ is an $\ell$-group for one
prime $\ell$.  If its exponent exceeded $\ell$, then $\ell\mathscr T$ would
be a nonzero proper characteristic subgroup.  Therefore $\mathscr T$ has
exponent $\ell$, and hence
\(\mathscr T\cong C_\ell^e\)
for some $e\geqslant1$.  Because it is the geometric kernel of a separable elliptic
isogeny, one has $e\leqslant2$; if $\ell=\operatorname{char}k$, reduced geometric
$\ell$-torsion has rank at most one, so $e=1$.  Hence
\(\deg g=\ell^e, \qquad e\in\{1,2\}, \qquad e=2\Longrightarrow \ell\ne\operatorname{char}k.\)

The conjugation action of $H$ on $\mathscr T$ is faithful: its kernel is
normal in $H$ and centralizes $\mathscr T$, hence is normal in
$G=\mathscr T\rtimes H$ and lies in $H$, which contains no nontrivial normal subgroup of $G$.  Finally, suppose that
\(0<W<\mathscr T\)
is stable under both the conjugation action of $H$ and the action induced by
$q$-power Frobenius.  By the construction above, this Frobenius action is
conjugation by $\widetilde{\operatorname{Fr}}_q$, so $W$ is stable under
$\Sigma$ as well.  Hence $WH$ is normalized by
\(H\Sigma=\widetilde H.\)
Since $\mathscr T\cap H=1$ and $0<W<\mathscr T$, one has
\(H<WH<G,\)
contradicting Proposition~\ref{prop:arithmetic-intermediate}.  Thus the group
generated by the induced $H$-action and Frobenius acts irreducibly on the
$\mathbb F_\ell$-space $\mathscr T$.
\end{proof}

\subsection{The remaining factor}

\begin{lemma}[Translations after quotienting by $\mathscr T$]
\label{lem:no-translations-after-translation-quotient}
The induced faithful action of $G/\mathscr T$ on $C/\mathscr T$ contains no
nontrivial translations.
\end{lemma}

\begin{proof}
Choose origins and let $\pi:C\to C/\mathscr T$ be the quotient isogeny.
If $a\in G$ induces a translation on $C/\mathscr T$, then the induced
linear part is the identity, so
\(\pi\circ(u_a-1)=0.\)
Thus $(u_a-1)(C)$ is contained in the finite kernel of $\pi$.  Since the
image of a connected algebraic group is connected, $u_a=1$.  Hence $a$ was
already a translation on $C$, so $a\in\mathscr T$.
\end{proof}

\begin{proposition}[Genus-zero Galois closures for the remaining factors]
In the factorization $g=h\circ g_{\mathrm{tr}}$ of
Theorem~\ref{thm:intrinsic-translation-factorization}, every $k$-indecomposable factor of $h$ has Galois closure of genus zero.  If $G/\mathscr T$ is
abelian, then $h$ itself has Galois closure of genus zero.  In particular, the latter conclusion holds when $\operatorname{char}k>3$.
\end{proposition}

\begin{proof}
Let a $k$-indecomposable factor of $h$ correspond to a subgroup interval
\(J\leq B_0<B_1\leq G\) arising from the corresponding $k$-factorization,
and put
\[
 \mathcal K_0:=\bigcap_{b\in B_1} bB_0b^{-1},
\]
the largest normal subgroup of $B_1$ contained in $B_0$.  Since
$\mathscr T\leq J\leq B_0$ and $\mathscr T\triangleleft G$, one has
$\mathscr T\leq \mathcal K_0$.  Its Galois-closure curve is
$C/\mathcal K_0$, a quotient of the genus-one curve $C/\mathscr T$, and
therefore has genus at most one.

Suppose it has genus one.  Riemann--Hurwitz with different for
$C/\mathscr T\to C/\mathcal K_0$ then forces zero different, so every nontrivial
element of $\mathcal K_0/\mathscr T$ acts without fixed points.  A nontranslation
automorphism $P\mapsto u(P)+a$ of a genus-one curve has a fixed point,
because $1-u$ is a nonzero isogeny and hence surjective.  Thus
$\mathcal K_0/\mathscr T$ consists of translations.  Lemma~\ref{lem:no-translations-after-translation-quotient}
forces $\mathcal K_0=\mathscr T$.  The factor would then be a $k$-indecomposable
genus-one map whose geometric monodromy has no nontrivial translations,
contradicting Proposition~\ref{prop:nontrivial-translation-subgroup}.  Hence its closure
has genus zero.

If $G/\mathscr T$ is abelian, then $J/\mathscr T$ is normal in
$G/\mathscr T$, so $J\triangleleft G$.  Therefore $C/J\to C/G$ is already
Galois, and its source $C/J$ is rational; hence the Galois closure of $h$ has
genus zero.  When $\operatorname{char}k>3$, the quotient
$G/\mathscr T$ embeds in the cyclic group of origin-preserving elliptic
curve automorphisms, so it is abelian.
\end{proof}

\section{Frobenius equations and rational points on elliptic quotients}
\label{sec:frobenius-equations-elliptic-quotients}

This section develops the rational-point argument underlying the exact kernel
criterion.  The equations $\mathcal F_E(P)=\gamma_E(P)$ describe rational
quotient fibres, and a weighted counting identity remains valid over ramified
fibres.

Let $E/k$ be an elliptic curve with a faithful action of a finite group
$\Lambda$ by curve automorphisms that is stable under Frobenius conjugation,
and assume that the image of the linear-part homomorphism is nontrivial.  Let
\(\pi:E\longrightarrow Y:=E/\Lambda\)
be the geometric quotient.  For $\gamma\in\Lambda$ write
$\gamma_E(P)=u_{\gamma,E}(P)+a_{\gamma,E}$.

\begin{lemma}[Descent and rationality of the quotient]
\label{lem:quotient}
The curve $Y$ and the morphism $\pi$ descend to $k$, and
$Y$ is $k$-isomorphic to $\mathbf P^1$.
\end{lemma}

\begin{proof}
Since the action is stable under Frobenius conjugation,
$\bar k(E)^\Lambda$ is Galois-stable, so the quotient and quotient morphism
descend to $k$; compare
\cite[Exercise~III.3.13(a),(e)]{Sil09}.  Choose
$\gamma\in\Lambda$ with $u_{\gamma,E}\ne1$.  The fixed-point equation for
$\gamma_E$ is
\((1-u_{\gamma,E})P=a_{\gamma,E}.\)
Since $1-u_{\gamma,E}$ is a nonzero isogeny, this equation has a geometric
solution.  Hence $E\to Y$ is ramified.  The Riemann--Hurwitz formula with
different \cite[Theorem~3.4.13]{Sti09} gives
\[
 0=\lvert\Lambda\rvert(2\operatorname{genus}(Y)-2)
 +\deg\operatorname{Diff}(\bar k(E)/\bar k(Y)),
\]
and the positive different forces $Y$ to have genus zero.  Finally,
$\pi(\mathrm O_E)\in Y(k)$, so $Y\cong_k\mathbf P^1$.
\end{proof}

Write $\mathcal F_E:E\to E$ for the $q$-power Frobenius endomorphism and use
the same symbol for its induced bijection of $E(\bar k)$.  Although
$\mathcal F_E$ is not an isomorphism of curves, its action on geometric points
is bijective; throughout, $\mathcal F_E^{-1}$ denotes only the inverse of this
bijection on $E(\bar k)$.  On geometric points, Frobenius conjugates curve
automorphisms by $\alpha\mapsto\mathcal F_E\circ\alpha\circ\mathcal F_E^{-1}$.
Faithfulness and stability of the prescribed action determine a unique
$\vartheta\in\operatorname{Aut}(\Lambda)$ such that
\[
 \vartheta(\gamma)_E
 =\mathcal F_E\circ\gamma_E\circ\mathcal F_E^{-1}
 \qquad\text{for all }\gamma\in\Lambda,
\]
where the compositions involving $\mathcal F_E^{-1}$ are understood on
geometric points.
Recall that, for $\gamma\in\Lambda$,
\[
 \mathcal S_\gamma(E)
 =\{P\in E(\bar k):\mathcal F_E(P)=\gamma_E(P)\}.
\]
Equivalently,
\((\mathcal F_E-u_{\gamma,E})P=a_{\gamma,E}.\)
Since $\mathcal F_E-u_{\gamma,E}$ is a nonzero separable endomorphism,
$\mathcal S_\gamma(E)$ is nonempty, and
$\ker(\mathcal F_E-u_{\gamma,E})$ acts freely and transitively on it by
translations.
For $\gamma,\eta\in\Lambda$, define
$\gamma\sim_\vartheta\eta$ if
\(\eta=\vartheta(h)\gamma h^{-1} \qquad\text{for some }h\in\Lambda.\)
This is an equivalence relation on $\Lambda$.  Write $[\gamma]_\vartheta$ for the
$\sim_\vartheta$-equivalence class of $\gamma$, and put
\[
 C_\vartheta(\gamma)
 :=\{h\in\Lambda:
 \vartheta(h)\gamma h^{-1}=\gamma\}.
\]

\begin{lemma}[Transport under the $\Lambda$-action]
\label{lem:fixed-point-set-transport}
For $h,\gamma\in\Lambda$,
\[
 h\mathcal S_\gamma(E)
 =\mathcal S_{\vartheta(h)\gamma h^{-1}}(E).
\]
Consequently, elements in the same $\sim_\vartheta$-class give sets with the
same image in $Y$.
\end{lemma}

\begin{proof}
If $\mathcal F_E(P)=\gamma_E(P)$, then
\(\mathcal F_E(hP) =\vartheta(h)\mathcal F_E(P) =\vartheta(h)\gamma h^{-1}(hP).\)
The reverse inclusion follows from $h^{-1}$, and $\pi h=\pi$ gives the last
assertion.
\end{proof}

Let $\mathcal F_Y$ be $q$-power Frobenius on $Y$.  Since $\pi$ is defined
over $k$, $\mathcal F_Y\pi=\pi\mathcal F_E$.

\begin{proposition}[Rational points on the quotient]
Let $\mathcal R$ be a complete set of representatives for the
$\sim_\vartheta$-classes in $\Lambda$.  Then
\begin{equation}
\label{eq:quotient-rational-point-union}
 Y(k)=\bigcup_{\gamma\in\Lambda}\pi(\mathcal S_\gamma(E))
 =\bigcup_{\gamma\in\mathcal R}\pi(\mathcal S_\gamma(E)).
\end{equation}
If
$E^\circ=\{P:\operatorname{Stab}_\Lambda(P)=\{1\}\}$ and
$Y^\circ=\pi(E^\circ)$, then there is a natural bijection
\begin{equation}
\label{eq:free-decomp}
 Y^\circ(k)
 \longleftrightarrow
 \bigsqcup_{\gamma\in\mathcal R}
 \frac{\mathcal S_\gamma(E)\cap E^\circ}{C_\vartheta(\gamma)}.
\end{equation}
Finally, if $P\in\mathcal S_\gamma(E)$ and $Q=\pi(P)$, then
$Q\in\pi(\mathcal S_\eta(E))$ if and only if there is $h\in\Lambda$ with
\begin{equation}
\label{eq:overlap}
 h^{-1}\eta^{-1}\vartheta(h)\gamma
 \in\operatorname{Stab}_\Lambda(P).
\end{equation}
Thus distinct $\sim_\vartheta$-classes meet only over the branch locus.
\end{proposition}

\begin{proof}
If $P\in\mathcal S_\gamma(E)$ then
$\mathcal F_Y(\pi(P))=\pi(\gamma_E(P))=\pi(P)$.  Conversely, if
$Q\in Y(k)$ and $P\in\pi^{-1}(Q)$, then
$\pi(\mathcal F_E(P))=\pi(P)$, so
$\mathcal F_E(P)=\gamma_E(P)$ for some $\gamma\in\Lambda$ because a
geometric fibre is a single $\Lambda$-orbit.  This proves
\eqref{eq:quotient-rational-point-union}.

If $P\in\mathcal S_\gamma(E)$ and $hP\in\mathcal S_\eta(E)$, then
\(\eta hP=\mathcal F_E(hP) =\vartheta(h)\mathcal F_E(P)=\vartheta(h)\gamma P,\)
which is equivalent to \eqref{eq:overlap}.  On $E^\circ$ the stabilizer is
trivial; Lemma~\ref{lem:fixed-point-set-transport} then gives
\eqref{eq:free-decomp}.
\end{proof}

Let $Q\in Y(k)$, choose $P\in\pi^{-1}(Q)$, and write
$\Lambda_P:=\operatorname{Stab}_\Lambda(P)$.  The geometric fibre
$\pi^{-1}(Q)$ is a single $\Lambda$-orbit.  The stabilizers of its points are therefore conjugate in
$\Lambda$, so $\lvert\Lambda_P\rvert$ depends only on $Q$, and
orbit--stabilizer gives
\(\lvert\pi^{-1}(Q)\rvert =\frac{\lvert\Lambda\rvert}{\lvert\Lambda_P\rvert}.\)
For $\gamma\in\Lambda$, define
\(m_Q(\gamma) :=\lvert\mathcal S_\gamma(E)\cap\pi^{-1}(Q)\rvert.\)

\begin{proposition}[Weighted fibre-counting identity]
\label{prop:weighted-fibre-counting}
With the notation above, every point of $\pi^{-1}(Q)$ lies in exactly
$\lvert\Lambda_P\rvert$ of the sets $\mathcal S_\gamma(E)$, $\gamma\in\Lambda$.
Consequently,
\begin{equation}
\label{eq:fibre-counting-sum}
 \sum_{\gamma\in\Lambda}m_Q(\gamma)=\lvert\Lambda\rvert.
\end{equation}
The function $\gamma\mapsto m_Q(\gamma)$ is constant on every $\sim_\vartheta$-class.  Thus, for any complete set
$\mathcal R$ of representatives for these equivalence classes,
\begin{equation}
\label{eq:class-weighted-fibre-counting}
 \sum_{\gamma\in\mathcal R}
 \lvert[\gamma]_\vartheta\rvert\,m_Q(\gamma)
 =\lvert\Lambda\rvert.
\end{equation}
\end{proposition}

\begin{proof}
Choose $\gamma_0$ with $\mathcal F_E(P)=\gamma_0(P)$.  If $R=hP$, then
$R\in\mathcal S_\gamma(E)$ if and only if
$h^{-1}\gamma^{-1}\vartheta(h)\gamma_0\in\Lambda_P$.  Thus the set of $\gamma$ for which $R$ lies in $\mathcal S_\gamma(E)$ is
$\vartheta(h)\gamma_0\Lambda_Ph^{-1}$ and has size $\lvert\Lambda_P\rvert$.
Orbit--stabilizer gives
$\lvert\pi^{-1}(Q)\rvert=\lvert\Lambda\rvert/\lvert\Lambda_P\rvert$, proving \eqref{eq:fibre-counting-sum} by double
counting.  Lemma~\ref{lem:fixed-point-set-transport} identifies these intersections for elements in the same
$\sim_\vartheta$-class, which gives \eqref{eq:class-weighted-fibre-counting}.
\end{proof}

\begin{remark}
Equation~\eqref{eq:fibre-counting-sum} exactly compensates for the smaller
geometric orbit at a branch point: each lift then belongs to proportionally
more of the sets $\mathcal S_\gamma(E)$.
\end{remark}

\begin{lemma}[Multiplicity in a cyclic quotient fibre]
Assume that $\Lambda=\langle h\rangle$ is cyclic of order $m$ prime to
$\operatorname{char}k$ and that the Frobenius action satisfies
\(\vartheta(h)=h^q.\)
For $r\geqslant1$ write
\[
 \mathcal S_{r,\gamma}(E)
 :=\{P\in E(\bar k):\mathcal F_E^r(P)=\gamma_E(P)\}.
\]
Let $Q\in(E/\Lambda)(k_r)$, let
\(A_Q:=\operatorname{Stab}_\Lambda(P)\)
for $P\in\pi^{-1}(Q)$, and put
\(m_{r,Q}(\gamma):=\lvert\mathcal S_{r,\gamma}(E)\cap\pi^{-1}(Q)\rvert.\)
If $m_{r,Q}(\gamma)>0$, then
\[
 m_{r,Q}(\gamma)
 =\gcd\!\left(\frac{m}{\lvert A_Q\rvert},q^r-1\right)
 \mid\frac{m}{\lvert A_Q\rvert}\mid m.
\]
\end{lemma}

\begin{proof}
Choose $P\in\mathcal S_{r,\gamma}(E)\cap\pi^{-1}(Q)$.  For $b\in\Lambda$,
the point $bP$ remains in the same set $\mathcal S_{r,\gamma}(E)$ precisely
when
\(b^{-1}\vartheta^r(b)=b^{q^r-1}\in A_Q.\)
Thus $m_{r,Q}(\gamma)$ is the cardinality of the kernel of
\(\Lambda/A_Q\longrightarrow\Lambda/A_Q, \qquad bA_Q\longmapsto b^{q^r-1}A_Q.\)
Since $\Lambda/A_Q$ is cyclic of order $m/\lvert A_Q\rvert$, the kernel has the displayed
cardinality.
\end{proof}

\section{The exact finite-kernel permutation criterion}
\label{sec:isogenies}

We apply the fixed-point-set identities of
Section~\ref{sec:frobenius-equations-elliptic-quotients} to an equivariant
isogeny datum, converting bijectivity on quotient rational points into an
exact condition on the finite isogeny kernel.

Fix an equivariant isogeny datum over $k$.  For $\gamma\in\Lambda$ write
\[
 \gamma_E(P)=u_{\gamma,E}(P)+a_{\gamma,E},
 \qquad
 \gamma_{E'}(P')=u_{\gamma,E'}(P')+a_{\gamma,E'}.
\]
Equivariance means
\begin{equation}
\label{eq:equivariance}
 \varphi\circ\gamma_E=\gamma_{E'}\circ\varphi.
\end{equation}
Evaluating at $\mathrm O_E$ and then subtracting the translation parts gives
\begin{equation}
\label{eq:linear-equivariance}
 \varphi(a_{\gamma,E})=a_{\gamma,E'},
 \qquad
 \varphi\circ u_{\gamma,E}=u_{\gamma,E'}\circ\varphi.
\end{equation}

\begin{lemma}[Target action under Frobenius conjugation]
\label{lem:target-frobenius}
The prescribed $\Lambda$-action on $E'$ is stable under Frobenius
conjugation with the same induced automorphism $\vartheta$, and the image of
its linear-part homomorphism is nontrivial.  In particular,
\[
 \mathcal F_{E'}\circ\gamma_{E'}\circ\mathcal F_{E'}^{-1}
 =\vartheta(\gamma)_{E'}
 \qquad\text{for all }\gamma\in\Lambda.
\]
\end{lemma}

\begin{proof}
Conjugating \eqref{eq:equivariance} by Frobenius and using that $\varphi$ is
defined over $k$ gives
\[
 (\mathcal F_{E'}\gamma_{E'}\mathcal F_{E'}^{-1})\varphi
 =\varphi\,\vartheta(\gamma)_E
 =\vartheta(\gamma)_{E'}\varphi.
\]
Surjectivity of $\varphi$ gives the Frobenius assertion.  If the image of the target linear-part homomorphism were trivial, then for
every $\gamma$ one would have
$\varphi\circ(u_{\gamma,E}-1)=0$.  The image of the group homomorphism
$u_{\gamma,E}-1:E\to E$ would then lie in the finite group
$\ker\varphi$, hence would be trivial; thus $u_{\gamma,E}=1$ for all
$\gamma$, contrary to Definition~\ref{def:equivariant-isogeny-datum}.
\end{proof}

Let $\pi:E\to E/\Lambda$ and $\pi':E'\to E'/\Lambda$ be the quotient
maps.  Lemma~\ref{lem:quotient} applied to both actions gives quotient
lines over $k$, and equivariance yields a unique $k$-morphism
$f:E/\Lambda\to E'/\Lambda$ with
$\pi'\varphi=f\pi$.  Since $\deg\pi=\deg\pi'=\lvert\Lambda\rvert$,
$\deg f=\deg\varphi$.

Choose and fix $k$-isomorphisms from $E/\Lambda$ and $E'/\Lambda$ to
$\mathbf P^1$, and let $f\in k(X)$ denote the resulting rational-function
representative of the induced quotient map.  Then $\deg f=\deg\varphi$.
Changing the source coordinate composes $f$ on the right with an element of
$\operatorname{PGL}_2(k)$, whereas changing the target coordinate composes it
on the left with such an element.  Hence the $k$-M\"obius equivalence class of
$f$ is intrinsic to the equivariant isogeny datum.  The map $f$ is separable because $k(E)/k(E'/\Lambda)$ is separable and
$k(E/\Lambda)$ is an intermediate field.  We retain these coordinates and
this representative for the rest of the paper.

For a finite morphism $\psi:X\to Y$ of smooth curves over $k$ and
$R\in X(\bar k)$, write $e_R(\psi)$ for the ramification index of $\psi$ at $R$,
characterized by
\(\operatorname{ord}_R(t\circ\psi)=e_R(\psi)\)
for any local parameter $t$ at $\psi(R)$; compare \cite[Chap.~II, \S2]{Sil09}.
We call $\psi$ \emph{tame} if all its ramification indices are prime to
$\operatorname{char}k$; this is the curve-morphism form of
\cite[Definition~3.5.4]{Sti09}.

\begin{proposition}[Ramification and automatic tameness]
\label{prop:tameness}
For $P\in E(\bar k)$,
\[
 e_{\pi(P)}(f)
 =[
 \operatorname{Stab}_\Lambda(\varphi(P)):
 \operatorname{Stab}_\Lambda(P)].
\]
Every ramification index divides $\lvert L_E\rvert$, and hence also $\lvert\Lambda\rvert$.
In particular, if $\operatorname{char}k>3$, then $f$ is tame.
\end{proposition}

\begin{proof}
Equivariance gives
$\operatorname{Stab}_\Lambda(P)\leq
 \operatorname{Stab}_\Lambda(\varphi(P))$.  For a finite group quotient,
\[
 e_P(\pi)=\lvert\operatorname{Stab}_\Lambda(P)\rvert,
 \qquad
 e_{\varphi(P)}(\pi')=\lvert\operatorname{Stab}_\Lambda(\varphi(P))\rvert
\]
\cite[Exercise~III.3.13(b)]{Sil09}.  Since the separable isogeny $\varphi$ has ramification index $1$ at every
geometric point, multiplicativity in $\pi'\varphi=f\pi$ gives the formula.
Put
$A_P=\operatorname{Stab}_\Lambda(P)$ and
$B_P=\operatorname{Stab}_\Lambda(\varphi(P))$.  The linear-part map
$\ell_E:\Lambda\to L_E$ is injective on $B_P$.  Indeed, suppose that
$\gamma\in B_P$ and $u_{\gamma,E}=1$.  By
\eqref{eq:linear-equivariance} and the surjectivity of $\varphi$, one has
$u_{\gamma,E'}=1$.  Hence $\gamma_{E'}$ is translation by
$a_{\gamma,E'}=\varphi(a_{\gamma,E})$.  Since $\gamma$ fixes
$\varphi(P)$, this translation is trivial, so the target action of $\gamma$
is the identity.  Target faithfulness therefore gives $\gamma=1$.  Thus
$\ell_E$ is injective on $B_P$, and hence also on $A_P$.  Consequently
\(e_{\pi(P)}(f)=[B_P:A_P] =[\ell_E(B_P):\ell_E(A_P)] \mid \lvert L_E\rvert,\)
and therefore also $e_{\pi(P)}(f)\mid\lvert\Lambda\rvert$.  If
$p=\operatorname{char}k>3$, then
$L_E\leq\operatorname{Aut}_{\bar k}(E,\mathrm O_E)$ and
$p\nmid\lvert\operatorname{Aut}_{\bar k}(E,\mathrm O_E)\rvert$ by
\cite[Theorem~III.10.1]{Sil09}.  Hence $p\nmid\lvert L_E\rvert$, proving tameness.
\end{proof}

\begin{lemma}[Isogeny on the fixed-point sets]
\label{lem:isogeny-on-fixed-point-sets}
For every $\gamma\in\Lambda$,
$\varphi(\mathcal S_\gamma(E))\subseteq\mathcal S_\gamma(E')$.
Two points of $\mathcal S_\gamma(E)$ have the same image under $\varphi$ if
and only if their difference lies in
$\mathcal N\cap\ker(\mathcal F_E-u_{\gamma,E})$.
Hence this group acts freely and transitively on every nonempty fibre of the
restricted map.  Moreover,
\begin{equation}
\label{eq:fixed-point-set-cardinality}
 \lvert\mathcal S_\gamma(E)\rvert=\lvert\mathcal S_\gamma(E')\rvert.
\end{equation}
Thus the restricted map is bijective if and only if
$\mathcal N\cap\ker(\mathcal F_E-u_{\gamma,E})=\{\mathrm O_E\}$.
\end{lemma}

\begin{proof}
By \eqref{eq:linear-equivariance} and Frobenius-equivariance,
\begin{equation}
\label{eq:intertwine}
 (\mathcal F_{E'}-u_{\gamma,E'})\circ\varphi
 =\varphi\circ(\mathcal F_E-u_{\gamma,E}).
\end{equation}
Together with
$\varphi(a_{\gamma,E})=a_{\gamma,E'}$, this sends $\mathcal S_\gamma(E)$ to $\mathcal S_\gamma(E')$.  If
$P_1,P_2$ lie in $\mathcal S_\gamma(E)$, subtracting their defining
equations shows
$P_1-P_2\in\ker(\mathcal F_E-u_{\gamma,E})$, which gives the fibre
criterion.  The two endomorphisms in \eqref{eq:intertwine} are nonzero and
separable because their differentials are $-du_{\gamma,E}$ and
$-du_{\gamma,E'}$.  Taking degrees in \eqref{eq:intertwine} gives equal degrees, and the
corresponding geometric kernel acts freely and transitively on each set
$\mathcal S_\gamma$ by translations.  This proves \eqref{eq:fixed-point-set-cardinality} and the last assertion.
\end{proof}

\begin{lemma}[Injectivity forces the kernel condition]
\label{lem:injectivity-kernel-condition}
If
$f:(E/\Lambda)(k)\to(E'/\Lambda)(k)$ is injective, then
\[
 \mathcal N\cap\ker(\mathcal F_E-u_{\gamma,E})
 =\{\mathrm O_E\}
 \qquad\text{for every }\gamma\in\Lambda.
\]
\end{lemma}

\begin{proof}
If the source action is origin-preserving, the assertion is immediate.
Indeed, suppose that
\(D\ne\mathrm O_E,\quad D\in \mathcal N\cap\ker(\mathcal F_E-u_{\gamma,E})\)
for some $\gamma\in\Lambda$.  Then $u_{\gamma,E}=\gamma_E$ and
$D\in\mathcal S_\gamma(E)$, so both $\pi(D)$ and $\pi(\mathrm O_E)$ are
$k$-rational.  They are distinct, because every element of $\Lambda$ fixes
$\mathrm O_E$, whereas
\(f(\pi(D)) =\pi'(\varphi(D)) =\pi'(\mathrm O_{E'}) =f(\pi(\mathrm O_E)).\)
This contradicts injectivity.  We may therefore assume from now on that the
source action is not origin-preserving.

Let $T$ be the kernel of the linear-part homomorphism
$\Lambda\to L_E$.  It is $\vartheta$-stable, so its elements act by
translations through a Frobenius-stable subgroup $A\subset E(\bar k)$; on
$E'$ they act through $A'=\varphi(A)$.  Target faithfulness gives
$A\cap\mathcal N=\{\mathrm O_E\}$.  Hence, with
$q_A:E\to\bar E:=E/A$, the map $\varphi$ descends over $k$ to a separable
isogeny $\bar\varphi:\bar E\to\bar E':=E'/A'$ with
$\bar{\mathcal N}:=\ker\bar\varphi=q_A(\mathcal N)\cong\mathcal N$, and
the actions factor through $\bar\Lambda:=\Lambda/T$.  The induced source
linear-part map is faithful: triviality of the linear part of $\gamma T$
would give $(u_{\gamma,E}-1)(E)\subseteq A$, hence $u_{\gamma,E}=1$ by
connectedness, so $\gamma\in T$.  The induced target action is faithful as
well: if $\gamma T$ acts trivially on $\bar E'$, then
$(u_{\gamma,E'}-1)(E')\subseteq A'$, hence $u_{\gamma,E'}=1$ and, by
\eqref{eq:linear-equivariance}, $(u_{\gamma,E}-1)(E)\subseteq\mathcal N$;
connectedness again gives $\gamma\in T$.  Since $T$ is $\vartheta$-stable,
the induced source action remains stable under Frobenius conjugation, and its
linear image is nontrivial.  Thus
$(\bar E,\bar E',\bar\Lambda,\bar\varphi)$ is again an equivariant isogeny
datum over $k$.

Moreover $q_A$ intertwines Frobenius and the induced linear parts and is
injective on $\mathcal N$, so for every $\gamma\in\Lambda$ it identifies
\[
 \mathcal N\cap\ker(\mathcal F_E-u_{\gamma,E})
 \quad\text{with}\quad
 \bar{\mathcal N}\cap
 \ker(\mathcal F_{\bar E}-u_{\gamma T,\bar E}).
\]
Quotienting in stages identifies the lower map of this new datum with the
original $f$.  Hence injectivity of $f$ and failure of the kernel condition
are both preserved.  Via the explicit isomorphism
$q_A|_{\mathcal N}:\mathcal N\xrightarrow{\sim}\bar{\mathcal N}$, and after
relabeling the quotient datum as $(E,E',\Lambda,\varphi)$, we may therefore
assume that the linear-part homomorphism is faithful.

We use two elementary observations.  First, there is no nonzero
$D\in\mathcal N$ fixed by Frobenius and by every linear part.  Indeed,
translation by such a $D$ commutes with $\Lambda$ and, since $D$ is
Frobenius-fixed, is defined over $k$; hence it descends to a $k$-automorphism
$\bar\tau_D$ of $E/\Lambda\cong\mathbf P^1$ satisfying
$f\bar\tau_D=f$.  If $\bar\tau_D=1$, then $\tau_D$ belongs to the deck
group $\Lambda$ of $E\to E/\Lambda$.  Its linear part is trivial, so
faithfulness of the linear-part homomorphism forces $D=\mathrm O_E$,
a contradiction.  Thus $\bar\tau_D\ne1$.  A nonidentity $k$-automorphism
of $\mathbf P^1$ fixes at most two geometric points, whereas
$\lvert\mathbf P^1(k)\rvert=q+1\ge3$; hence it moves a $k$-rational point,
contradicting injectivity.

Second, suppose the resulting group $\Lambda$ is cyclic and put
$K_\gamma=\mathcal N\cap\ker(\mathcal F_E-u_{\gamma,E})$.  If
$K_\gamma\ne\{\mathrm O_E\}$, choose $P\in\mathcal S_\gamma(E)$ and set
$B=\operatorname{Stab}_\Lambda(P)$.  Since $\Lambda$ is abelian, $B$ is
$\vartheta$-stable, and the overlap formula
\eqref{eq:overlap} gives
\begin{equation}
\label{eq:hiding-divisibility}
 \lvert K_\gamma\rvert
 \mid\lvert\mathcal S_\gamma(E)\cap\pi^{-1}(\pi(P))\rvert
 =\lvert(\Lambda/B)^\vartheta\rvert
 \mid\lvert\Lambda\rvert.
\end{equation}
Here the first divisibility follows because $K_\gamma$ acts freely by
translations on the displayed fibre intersection; injectivity of $f$ forces all
these translates into the same quotient fibre.

With the translation kernel of the linear-part action removed, we rule out
nonzero points violating the kernel condition first in characteristic greater
than $3$, and then in the exceptional small-characteristic automorphism
groups.

Assume first that $\operatorname{char}k>3$.  Then the image of the faithful
linear-part homomorphism is $C_m$ with $m\in\{2,3,4,6\}$
\cite[Theorem~III.10.1]{Sil09}.  For $C_2$, a nonzero $K_\gamma$ has order
$2$ by \eqref{eq:hiding-divisibility}; its unique nonzero element is fixed
by Frobenius and by every element of this cyclic image, contradicting the
first observation.  For $C_3$, \eqref{eq:hiding-divisibility} forces
$\lvert K_\gamma\rvert=3$.  The middle term in
\eqref{eq:hiding-divisibility} therefore also has size $3$, which forces
$B=\operatorname{Stab}_\Lambda(P)=1$.  Hence the preceding fibre argument
gives $P+K_\gamma\subseteq\Lambda\cdot P$, and both sets have three
points, so they are equal.  On this common three-point set, the translations
by $K_\gamma$ and the action of $\Lambda\cong C_3$ are regular; hence their
permutation images are the same unique subgroup of order $3$ in $S_3$.
If $\lambda$ generates $\Lambda$, there is therefore some
$D_0\in K_\gamma$ such that $\lambda$ restricts to translation by $D_0$ on
$P+K_\gamma$.  For every $D\in K_\gamma$, comparing the images of $P+D$
and $P$ gives
$u_{\lambda,E}(D)=D$.  Thus every linear part fixes $K_\gamma$ pointwise.
The defining relation
$\mathcal F_E(D)=u_{\gamma,E}(D)$ then makes every $D\in K_\gamma$
Frobenius-fixed, contradicting the first observation for $D\ne\mathrm O_E$.
We also use the following reduction.  If the source action has a common
fixed point $P_0\in E(k)$, then $\varphi(P_0)\in E'(k)$ is fixed by the
target action.  Conjugating the source and target actions by translations by
$-P_0$ and $-\varphi(P_0)$ makes both actions origin-preserving.  Since
$\varphi$ is a group homomorphism,
$\tau_{-\varphi(P_0)}\circ\varphi\circ\tau_{P_0}=\varphi$; because the two
points are $k$-rational, these conjugations also preserve Frobenius, the
linear parts, $\mathcal N$, the kernel condition, and injectivity of the
induced map on quotient $k$-points.  The initial origin-preserving argument therefore applies after this reduction.

For $C_4$, \eqref{eq:hiding-divisibility} gives
$\lvert K_\gamma\rvert\in\{2,4\}$.  Choose
$D\ne\mathrm O_E$ in $K_\gamma[2]$ and let $u$ generate the linear image.
If $uD\ne D$, replace $D$ by $D+uD$; in either case we obtain a nonzero
$\delta\in\mathcal N[2]$ with $u\delta=\delta$.  Since $u^2=-1$ and
$1+u=u(1-u)$,
\[
 (1-u)(1+u)=[2],\qquad \deg(1-u)=2.
\]
Here $1-u$ is separable, so its kernel consists of $\mathrm O_E$ and the
unique nonzero point $\delta$.  Frobenius conjugation sends $u$ to
$u^{\pm1}$, and $\ker(1-u^{-1})=\ker(1-u)$; hence Frobenius preserves this
two-point kernel and fixes $\delta$.  Since $u$ generates the linear image,
$\delta$ is fixed by every linear part as well, contradicting the first
observation.

For $C_6$, let $\lambda$ generate $\Lambda$ and put
$u=u_{\lambda,E}$.  Then $u^2-u+1=0$, so $1-u=-u^2$ is an automorphism.
Thus the affine automorphism $\lambda_E$ has a unique geometric fixed point
$P_0$.  Stability gives
$\mathcal F_E\lambda_E\mathcal F_E^{-1}=\vartheta(\lambda)_E$, and
$\vartheta(\lambda)$ is again a generator of $\Lambda$; hence
$\mathcal F_E(P_0)$ is fixed by a generator and therefore equals $P_0$.
Thus $P_0\in E(k)$, and because $\lambda$ generates $\Lambda$ it is a common
fixed point.  The common-fixed-point reduction gives the required contradiction.

It remains to treat the enlarged origin-preserving automorphism groups in
characteristics $3$ and $2$.  For $j=0$, Silverman gives respectively a group
of order $12$ with a normal $C_3$, and a group of order $24$ with a normal
quaternion subgroup; see \cite[Appendix~A, Proposition~A.1.2(c) and
Exercise~A.1]{Sil09}.  In characteristic $3$, suppose first that
$3\mid\lvert\Lambda\rvert$.  The normal $C_3$ is the unique Sylow
$3$-subgroup of the full origin-preserving automorphism group, hence is
contained in the faithful linear image of $\Lambda$.  Let $\lambda$ be a
nontrivial element of its preimage and put $u=u_{\lambda,E}$.  In the standard
$j=0$ form, $u$ acts by a nonzero translation of the $x$-coordinate and fixes
only $\mathrm O_E$.  Thus $\ker(1-u)(\bar k)=\{\mathrm O_E\}$; since
$1-u$ is a nonzero isogeny, the affine automorphism $\lambda_E$ has a unique
geometric fixed point $P_0$.  Normality of the $C_3$ implies that every
$h\in\Lambda$ conjugates $\lambda$ to $\lambda^{\pm1}$, so $h(P_0)$ is again
fixed by $\lambda$ and hence equals $P_0$.  Stability under Frobenius gives
the same conclusion for $\mathcal F_E(P_0)$.  Therefore $P_0$ is a common
$k$-rational fixed point, and the common-fixed-point reduction applies.  If
$3\nmid\lvert\Lambda\rvert$, then $\lvert\Lambda\rvert$ is $2$ or $4$;
Lemma~\ref{lem:small-char-involution} excludes a Klein four subgroup, so the
$C_2$ or $C_4$ argument above applies verbatim in characteristic $3$.

In characteristic $2$, the full $j=0$ group has the unique involution
$[-1]$ by Lemma~\ref{lem:small-char-involution}; in the standard form it
fixes only $\mathrm O_E$.  If $\lvert\Lambda\rvert$ is even, the unique
involution $\lambda\in\Lambda$ has this linear part and is central.  Since
$1-[-1]=[2]$ is a nonzero isogeny with trivial geometric kernel,
$\lambda_E$ has a unique geometric fixed point $P_0$.  Centrality makes
$P_0$ fixed by all of $\Lambda$, while uniqueness of the involution and
Frobenius stability give $\vartheta(\lambda)=\lambda$ and hence
$\mathcal F_E(P_0)=P_0$.  The common-fixed-point reduction therefore applies.
If $\lvert\Lambda\rvert$ is odd, nontriviality of the linear image and
$\lvert\Lambda\rvert\mid24$ force $\Lambda\cong C_3$, so the cyclic
$C_3$ argument above applies verbatim in characteristic $2$.  This completes
the proof.
\end{proof}

\begin{theorem}[Kernel criterion on $\sim_\vartheta$-classes]
\label{thm:kernel}
The induced map
$f:(E/\Lambda)(k)\to(E'/\Lambda)(k)$ is bijective if and only if
\begin{equation}
\label{eq:kernel}
 \mathcal N\cap\ker(\mathcal F_E-u_{\gamma,E})
 =\{\mathrm O_E\}
\end{equation}
for $\gamma$ in a complete set of representatives for the
$\sim_\vartheta$-classes of $\Lambda$.
\end{theorem}

\begin{proof}
The condition is invariant on each $\sim_\vartheta$-class.  Indeed, if
$\eta=\vartheta(h)\gamma h^{-1}$ and
$\mathcal F_E(D)=u_{\gamma,E}(D)$, then
$\mathcal F_E(u_{h,E}D)=u_{\eta,E}(u_{h,E}D)$; the linear action preserves
$\mathcal N$.  Necessity is Lemma~\ref{lem:injectivity-kernel-condition}.

Conversely, assume \eqref{eq:kernel}; by invariance it holds for every
$\gamma$.  Lemma~\ref{lem:isogeny-on-fixed-point-sets} gives bijections
$\mathcal S_\gamma(E)\xrightarrow{\sim}\mathcal S_\gamma(E')$.
For $Q'\in(E'/\Lambda)(k)$ put
$m'_{Q'}(\gamma):=\lvert\mathcal S_\gamma(E')\cap(\pi')^{-1}(Q')\rvert$.
Since $\pi'\varphi=f\pi$,
\[
 m'_{Q'}(\gamma)
 =\sum_{\substack{Q\in(E/\Lambda)(k)\\ f(Q)=Q'}}m_Q(\gamma).
\]
Applying \eqref{eq:fibre-counting-sum} first on $E'$ and then on every
source fibre gives
\[
 \begin{aligned}
 \lvert\Lambda\rvert
 &=\sum_{\gamma\in\Lambda}m'_{Q'}(\gamma)\\
 &=\sum_{\substack{Q\in(E/\Lambda)(k)\\ f(Q)=Q'}}
   \sum_{\gamma\in\Lambda}m_Q(\gamma)\\
 &=\bigl\lvert\{Q\in(E/\Lambda)(k):f(Q)=Q'\}\bigr\rvert
   \,\lvert\Lambda\rvert.
 \end{aligned}
\]
Thus every rational target point has exactly one rational preimage.
\end{proof}

Frobenius conjugation preserves $L_E$, the image of the linear-part
homomorphism.  We write
\(\vartheta_L(u):=\mathcal F_E\circ u\circ\mathcal F_E^{-1}\) for all \(u\in L_E\)
for the induced automorphism of $L_E$.  Thus
\(\vartheta_L(u_{\gamma,E}) =u_{\vartheta(\gamma),E}.\)
For $u,v\in L_E$, define $u\sim_{\vartheta_L}v$ by
\(v=\vartheta_L(h)uh^{-1} \qquad\text{for some }h\in L_E.\)

\begin{corollary}[Multiplication maps]
Suppose $\varphi=[n]$ with $\gcd(n,q)=1$, put $L=L_E$, and let
$\mathcal R_L$ be a complete set of representatives for the
$\sim_{\vartheta_L}$-classes of $L$.  Then $\mathcal N=E[n]$, and
the quotient map is a permutation if and only if
\[
 \prod_{u\in\mathcal R_L}
 \det(\mathcal F_E-u\mid E[n])
 \in(\mathbb Z/n\mathbb Z)^\times.
\]
Equivalently, one may take the product over all $u\in L$.
\end{corollary}

\begin{proof}
The module $E[n]$ is free of rank two over $\mathbb Z/n\mathbb Z$, and an
endomorphism is invertible exactly when its determinant is a unit.
Theorem~\ref{thm:kernel} is equivalently the requirement that
$\mathcal F_E-u$ be invertible on $E[n]$ for every $u\in L$.

This condition is constant on $\sim_{\vartheta_L}$-classes.
Indeed, if $v=\vartheta_L(h)uh^{-1}$ with $h\in L$, then on $E[n]$,
\((\mathcal F_E-v)h =\vartheta_L(h)(\mathcal F_E-u).\)
Hence $\mathcal F_E-v$ is invertible if and only if $\mathcal F_E-u$ is
invertible.  It therefore suffices to test $u\in\mathcal R_L$.  Finally, a
finite product in $\mathbb Z/n\mathbb Z$ is a unit if and only if every
factor is a unit.
\end{proof}

\section{Affine monodromy, decomposition, and permutation extension degrees}
\label{sec:extensions}

The criterion of Section~\ref{sec:isogenies} depends only on
$\mathcal N$, $\Lambda_{\mathcal N}$, and $\mathcal F_{\mathcal N}$.  These
data also determine monodromy after base change, decomposition classes, and all
permutation extension degrees.  From this point, when
$\mathcal N$ is viewed as an additive finite module, we write $0$ for its
identity point $\mathrm O_E$.

For $r\geqslant1$ put $k_r=\mathbb F_{q^r}$.  Applying
Theorem~\ref{thm:kernel} over $k_r$ gives the following.

\begin{corollary}[Extension-field kernel criterion]
\label{cor:extension-kernel}
For every $r\geqslant1$, the map
$f:(E/\Lambda)(k_r)\to(E'/\Lambda)(k_r)$ is bijective if and only if
\[
 \mathcal N\cap\ker(\mathcal F_E^r-u_{\gamma,E})=\{\mathrm O_E\}
\]
for a complete set of representatives for the equivalence relation defined by
$\eta=\vartheta^r(h)\gamma h^{-1}$ on $\Lambda$.  Equivalently, the condition may be imposed
for every $\gamma\in\Lambda$, or every $u\in L_E$.
\end{corollary}

Since $\varphi$ is defined over $k$, $\mathcal N$ is stable under Frobenius.  Recall the notation introduced above:
\[
 \mathcal F_{\mathcal N}=\mathcal F_E|_{\mathcal N},
 \qquad
 \Lambda_{\mathcal N}=
 \operatorname{im}\bigl(
 \Lambda\to\operatorname{Aut}(\mathcal N),\
 \gamma\mapsto u_{\gamma,E}|_{\mathcal N}\bigr),
\]
\[
 \mathcal A_{\mathcal N}
 =\langle\Lambda_{\mathcal N},\mathcal F_{\mathcal N}\rangle,
 \qquad
 \mathcal C_{\mathcal N}
 =\mathcal A_{\mathcal N}/\Lambda_{\mathcal N},
 \qquad
 \overline{\mathcal F}_{\mathcal N}
 =\mathcal F_{\mathcal N}\Lambda_{\mathcal N}.
\]
Frobenius normalizes $\Lambda_{\mathcal N}$, and
$\mathcal C_{\mathcal N}$ is the cyclic group generated by
$\overline{\mathcal F}_{\mathcal N}$.  Recall also that
\[
 c_{\mathcal N}
 =\operatorname{ord}_{\mathcal C_{\mathcal N}}
 (\overline{\mathcal F}_{\mathcal N})
 =\min\{\nu\geqslant1:
 \mathcal F_{\mathcal N}^{\,\nu}\in\Lambda_{\mathcal N}\}.
\]

For the fixed representative $f\in k(X)$ chosen in
Section~\ref{sec:isogenies}, let $\Omega_f$ and
$\kappa_f=\Omega_f\cap\bar k$ be its arithmetic Galois closure and full
constant field, as in the Introduction.  Put
\[
 A_f:=\operatorname{Gal}(\Omega_f/k(f(\mathbf x))),
 \qquad
 G_f:=\operatorname{Gal}(\Omega_f/\kappa_f(f(\mathbf x))).
\]

\begin{theorem}[Affine monodromy realization]
\label{thm:affine-monodromy}
There are permutation-group isomorphisms
\[
 G_f\cong\mathcal N\rtimes\Lambda_{\mathcal N},
 \qquad
 A_f\cong\mathcal N\rtimes\mathcal A_{\mathcal N}.
\]
Consequently
\[
 A_f/G_f\cong\mathcal C_{\mathcal N},
 \qquad
 [\kappa_f:k]=c_{\mathcal N}.
\]
\end{theorem}

\begin{proof}
Over $\bar k$, the composite
$\pi'\varphi:E\to E'/\Lambda$ is Galois.  Let $\tau_Q$ denote translation
by $Q\in\mathcal N$.  The translations $\mathcal N$ and the prescribed
affine action of $\Lambda$ preserve $\pi'\varphi$, and
\(\gamma_E\tau_Q\gamma_E^{-1}=\tau_{u_{\gamma,E}(Q)}.\)
Their intersection is trivial: if $\tau_Q$ is also the action of
$\gamma\in\Lambda$, then the target action of $\gamma$ is trivial because
$Q\in\ker\varphi$, contrary to target faithfulness unless $Q=0$ and
$\gamma=1$.  Hence the automorphism group of this Galois cover is
$\mathcal H=\mathcal N\rtimes\Lambda$ and has the required degree.
The intermediate cover $E/\Lambda\to E'/\Lambda$ corresponds to
$\Lambda\leq\mathcal H$.

Let
$K=\ker(\Lambda\to\operatorname{Aut}(\mathcal N))$, with the map given by
linear parts.  Then $K\triangleleft\mathcal H$.  Conversely, if $\gamma$ lies in the largest
normal subgroup of $\mathcal H$ contained in $\Lambda$, then for every
$Q\in\mathcal N$,
\(\tau_Q\gamma\tau_Q^{-1} =\tau_{Q-u_{\gamma,E}(Q)}\gamma\in\Lambda.\)
The trivial intersection above forces $u_{\gamma,E}(Q)=Q$ for every $Q$,
so $\gamma\in K$.  Thus
\(G_f\cong(\mathcal N\rtimes\Lambda)/K \cong\mathcal N\rtimes\Lambda_{\mathcal N}.\)
The geometric sheets are canonically identified with $\mathcal N$ via the
bijection
\[
 \mathcal N\longrightarrow\mathcal H/\Lambda,
 \qquad
 Q\longmapsto\tau_Q\Lambda.
\]
Let $\sigma$ be the permutation of these sheets induced by $q$-power
semilinear Frobenius.  Frobenius stability of $\mathcal N$ and of the
prescribed $\Lambda$-action makes $\sigma$ normalize the geometric
monodromy action and fix the zero sheet $\Lambda$.  Moreover
\[
 \mathcal F_E\tau_Q\mathcal F_E^{-1}
 =\tau_{\mathcal F_{\mathcal N}(Q)},
\]
so $\sigma$ acts on $\mathcal N$ as the linear map
$\mathcal F_{\mathcal N}$.  The constant-field exact sequence
\[
 1\longrightarrow G_f\longrightarrow A_f
 \longrightarrow\operatorname{Gal}(\kappa_f/k)\longrightarrow1
\]
has cyclic quotient generated by the image of $q$-power Frobenius.  Hence
$A_f$ is generated, in its sheet action, by $G_f$ and $\sigma$, and therefore
\[
 A_f\cong
 \mathcal N\rtimes
 \langle\Lambda_{\mathcal N},\mathcal F_{\mathcal N}\rangle
 =\mathcal N\rtimes\mathcal A_{\mathcal N}.
\]
Taking the quotient by $G_f$ gives
$A_f/G_f\cong\mathcal A_{\mathcal N}/\Lambda_{\mathcal N}
=\mathcal C_{\mathcal N}$.  Comparing with the same constant-field exact
sequence yields
$[\kappa_f:k]=|\mathcal C_{\mathcal N}|=c_{\mathcal N}$.
\end{proof}

\begin{remark}
The affine translation components of the prescribed symmetry action do not
appear in this permutation representation: both $G_f$ and $A_f$ depend on the
symmetry only through its induced linear action on $\mathcal N$.  Compare
the characteristic-zero cyclic affine-monodromy calculation in
\cite[Lemma~6.10]{GMS03}.  Theorem~\ref{thm:affine-monodromy}
gives the finite-field form, compatible with Frobenius conjugation, used
here; in particular,
\(A_f/G_f\cong\mathcal C_{\mathcal N}\)
is the cyclic quotient governing arithmetic monodromy after constant-field
extension.
\end{remark}

\begin{corollary}[Base change of affine monodromy]
\label{cor:base-change-monodromy}
For $r\geqslant1$, put
\[
 \Omega_{f,r}:=\Omega_f k_r,
 \qquad
 \kappa_{f,r}:=\Omega_{f,r}\cap\bar k,
 \qquad
 \mathcal A_{\mathcal N}^{(r)}
 :=\langle\Lambda_{\mathcal N},\mathcal F_{\mathcal N}^{\,r}\rangle.
\]
Then $\Omega_{f,r}$ is the arithmetic Galois closure of
$k_r(\mathbf x)/k_r(f(\mathbf x))$.  Its arithmetic monodromy group is
\[
 A_f^{(r)}
 :=\operatorname{Gal}(\Omega_{f,r}/k_r(f(\mathbf x)))
 \cong \mathcal N\rtimes\mathcal A_{\mathcal N}^{(r)},
\]
while, under the common geometric-sheet identification, its geometric
monodromy group is canonically identified with
$G_f\cong\mathcal N\rtimes\Lambda_{\mathcal N}$.
Moreover,
\[
 \kappa_f=\mathbb F_{q^{c_{\mathcal N}}},
 \qquad
 \kappa_{f,r}=\kappa_f k_r
 =\mathbb F_{q^{\operatorname{lcm}(c_{\mathcal N},r)}},
\]
and consequently
\[
 [A_f^{(r)}:G_f]
 =[\kappa_{f,r}:k_r]
 =\frac{c_{\mathcal N}}{\gcd(c_{\mathcal N},r)}.
\]
In particular, $c_{\mathcal N}\mid r$ if and only if $A_f^{(r)}=G_f$.
If $\gcd(r,c_{\mathcal N})=1$, then
$\mathcal A_{\mathcal N}^{(r)}=\mathcal A_{\mathcal N}$.
\end{corollary}

\begin{proof}
Write $f=A/B$ with coprime $A,B\in k[X]$ and set
$P_f(T):=A(T)-f(\mathbf x)B(T)$.  Up to multiplication by a nonzero element of $k(f(\mathbf x))$, this is a
minimal polynomial of $\mathbf x$ over $k(f(\mathbf x))$: it vanishes at
$\mathbf x$ and has degree $\deg f=[k(\mathbf x):k(f(\mathbf x))]$.
For every $r\geqslant1$ one likewise has
$[k_r(\mathbf x):k_r(f(\mathbf x))]=\deg f$, so the same polynomial remains irreducible, hence minimal, over
$k_r(f(\mathbf x))$.  Since $f$ is separable, $\Omega_f$ is the splitting
field of $P_f(T)$ over $k(f(\mathbf x))$; therefore the splitting field over
$k_r(f(\mathbf x))$ is exactly $\Omega_f k_r=\Omega_{f,r}$.  This proves the
asserted identity of arithmetic Galois closures.

Because $k_r\subseteq\bar k$, one has
$\Omega_{f,r}\bar k=\Omega_f\bar k$, so the geometric Galois closure and
its sheet action are unchanged after the constant extension.  Arithmetic
Frobenius over $k_r$ acts on $\mathcal N$ as
$\mathcal F_{\mathcal N}^{\,r}$.  Applying
Theorem~\ref{thm:affine-monodromy} over $k_r$ therefore gives
\[
 A_f^{(r)}\cong
 \mathcal N\rtimes
 \langle\Lambda_{\mathcal N},\mathcal F_{\mathcal N}^{\,r}\rangle
\]
and
\[
 A_f^{(r)}/G_f
 \cong
 \left\langle\overline{\mathcal F}_{\mathcal N}^{\,r}\right\rangle
 \leq\mathcal C_{\mathcal N}.
\]
Hence
\[
 [A_f^{(r)}:G_f]
 =\operatorname{ord}\!\left(\overline{\mathcal F}_{\mathcal N}^{\,r}\right)
 =\frac{c_{\mathcal N}}{\gcd(c_{\mathcal N},r)}.
\]
Theorem~\ref{thm:affine-monodromy} gives
$\kappa_f=\mathbb F_{q^{c_{\mathcal N}}}$.  The constant-field exact sequence
over $k_r$ identifies the preceding index with
$[\kappa_{f,r}:k_r]$.  Since $\kappa_f k_r\subseteq\kappa_{f,r}$ and
\[
 [\kappa_f k_r:k_r]
 =\frac{c_{\mathcal N}}{\gcd(c_{\mathcal N},r)},
\]
the two fields are equal; hence
\[
 \kappa_{f,r}=\kappa_f k_r
 =\mathbb F_{q^{\operatorname{lcm}(c_{\mathcal N},r)}}.
\]
Finally, if $\gcd(r,c_{\mathcal N})=1$, then
$\overline{\mathcal F}_{\mathcal N}^{\,r}$ generates
$\mathcal C_{\mathcal N}$.  Thus
$\mathcal A_{\mathcal N}^{(r)}$ contains $\Lambda_{\mathcal N}$ and maps
onto
$\mathcal A_{\mathcal N}/\Lambda_{\mathcal N}
=\mathcal C_{\mathcal N}$, so
$\mathcal A_{\mathcal N}^{(r)}=\mathcal A_{\mathcal N}$.
\end{proof}

\begin{corollary}[Extension degrees coprime to the constant-field degree]
Assume that $f$ is exceptional over $k$.  If
\(\gcd(r,c_{\mathcal N})=1,\)
then $f/k_r$ is exceptional and
\(f:\mathbf P^1(k_r)\longrightarrow\mathbf P^1(k_r)\)
is bijective.
\end{corollary}

\begin{proof}
By Corollary~\ref{cor:base-change-monodromy},
$\mathcal A_{\mathcal N}^{(r)}=\mathcal A_{\mathcal N}$ when
$\gcd(r,c_{\mathcal N})=1$.  Hence, under the common identification of
the geometric sheets, the arithmetic and geometric permutation groups
governing exceptionality are the same over $k_r$ as over $k$.  Thus
exceptionality over $k$ implies exceptionality over $k_r$.

A collision between two distinct $k_r$-rational points would persist over
every finite extension of $k_r$, contradicting exceptionality.  Hence $f$
is injective, and therefore bijective, on $\mathbf P^1(k_r)$.
\end{proof}

\begin{proposition}[Decomposition criterion]
\label{prop:decomposition}
Equivalence classes of decompositions
\[
 f=f_s\circ\cdots\circ f_1
\]
over $k$ are in bijection with strict chains
\[
 \{0\}=\mathcal N_0<\mathcal N_1<\cdots<\mathcal N_s=\mathcal N
\]
of $\mathcal A_{\mathcal N}$-stable subgroups.  Under this correspondence,
\[
 \deg f_i=[\mathcal N_i:\mathcal N_{i-1}].
\]
In particular,
\[
 f\text{ is }k\text{-indecomposable}
 \Longleftrightarrow
 \mathcal N\text{ has no nonzero proper }
 \mathcal A_{\mathcal N}\text{-stable subgroup}.
\]
\end{proposition}

\begin{proof}
If $\mathcal N=0$, then $\deg f=1$: there is no $k$-decomposition with
$s\geqslant1$ and no strict chain of the displayed form with $s\geqslant1$,
while the final indecomposability criterion is immediate under our convention.
Assume $\mathcal N\ne0$.  Apply \cite[Lemma~2.8]{KZ14} to the arithmetic
monodromy action
\(A_f\cong\mathcal N\rtimes\mathcal A_{\mathcal N}.\)
For covers of $\mathbf P^1$, the decomposition equivalence in that lemma is
exactly the rational-function equivalence used here; compare
\cite[Example~2.7]{KZ14}.  The stabilizer of the sheet $0$ is
$\mathcal A_{\mathcal N}$, while $\mathcal N$ acts regularly on the sheets.
Thus an intermediate subgroup $U$ with
$\mathcal A_{\mathcal N}\leq U\leq A_f$ corresponds to
$J:=U\cap\mathcal N$, with $U=J\mathcal A_{\mathcal N}$.  Since
$\mathcal N\cap\mathcal A_{\mathcal N}=0$, the condition
$J\mathcal A_{\mathcal N}=\mathcal A_{\mathcal N}J$ is equivalent to
$\mathcal A_{\mathcal N}$-stability of $J$.  Reversing the resulting decreasing
chain gives the displayed increasing chain
$0=\mathcal N_0<\cdots<\mathcal N_s=\mathcal N$.  For the corresponding
intermediate groups $U_i=\mathcal N_i\rtimes\mathcal A_{\mathcal N}$,
\([U_i:U_{i-1}]=[\mathcal N_i:\mathcal N_{i-1}],\)
so \cite[Lemma~2.8]{KZ14} gives the degree formula.  The final criterion is
the case of a two-factor decomposition.
\end{proof}

\begin{remark}[Elliptic realization and composition factors]
\label{rem:decomposition-realization}
This correspondence also has an elliptic realization by equivariant quotient
isogenies.
If $J\leq\mathcal N$ is $\mathcal A_{\mathcal N}$-stable, then it is
stable under Frobenius, so it defines a reduced finite $k$-subgroup of $E$ and a
separable quotient isogeny
\(E\longrightarrow E/J.\)
The affine $\Lambda$-action descends because its linear parts preserve $J$.
The descended $\Lambda$-action on $E/J$ remains faithful.  Indeed, $\varphi$
factors as
\[
 E\longrightarrow E/J\xrightarrow{\bar\varphi}E',
\]
with $\bar\varphi$ surjective.  If $\gamma\in\Lambda$ acts trivially on
$E/J$, equivariance gives
\(\gamma_{E'}\circ\bar\varphi=\bar\varphi,\)
so $\gamma_{E'}=1$ by surjectivity.  The prescribed target action is
faithful, hence $\gamma=1$.

The image of the descended linear-part homomorphism is also nontrivial.
Otherwise
\((u_{\gamma,E}-1)(E)\subseteq J\)
for every $\gamma\in\Lambda$.  Since $E$ is connected and $J$ is finite,
this forces $u_{\gamma,E}=1$ for every $\gamma$, contrary to the defining
nontriviality of the source linear-part image.  Thus every chain in Proposition~\ref{prop:decomposition} may be
represented by a chain of equivariant quotient isogenies.  More precisely, for
$0=\mathcal N_0<\cdots<\mathcal N_s=\mathcal N$, the successive map
\(E/\mathcal N_{i-1}\to E/\mathcal N_i\)
is a separable $k$-isogeny with geometric kernel
\(\mathcal N_i/\mathcal N_{i-1}.\)
At the final stage $E/\mathcal N$ is $k$-isomorphic to $E'$ via the
isomorphism induced by $\varphi$.

Complete decompositions therefore correspond to composition series of the
finite $\mathbb Z[\mathcal A_{\mathcal N}]$-module $\mathcal N$.  The
multiset of simple composition factors is independent of the chosen composition
series.  Accordingly, the corresponding invariance in the factorization theorem of
Kre\v{s}o--Zieve follows here from the composition factors of the finite kernel;
we do not claim a new general factorization theorem.
\end{remark}

\begin{corollary}[Prime and prime-square factor degrees]
\label{cor:prime-prime-square-factor-degrees}
Every $k$-indecomposable factor of such an induced quotient map has degree $\ell$ or $\ell^2$ for some prime $\ell$.  If
$\ell=\operatorname{char}k$, only degree $\ell$ can occur.
\end{corollary}

\begin{proof}
Let
\(V_i=\mathcal N_i/\mathcal N_{i-1}\)
be a successive quotient in a composition series from
Proposition~\ref{prop:decomposition}.  It has no nonzero proper invariant
subgroup, hence, as a finite abelian group, it has no nontrivial proper characteristic
subgroup:
\(V_i\cong C_\ell^e\)
for a prime $\ell$.  Remark~\ref{rem:decomposition-realization} identifies $V_i$ with the
geometric kernel of a separable isogeny between elliptic curves.  Thus
$e\leqslant2$; in characteristic $\ell$, the reduced geometric $\ell$-torsion
has rank at most one, so $e=1$.  The degree formula in
Proposition~\ref{prop:decomposition} gives the result.
\end{proof}

\begin{remark}[Monodromy of the factors]
For $V_i=\mathcal N_i/\mathcal N_{i-1}$, let $\Lambda_i$ and
$\mathcal A_i$ be the induced images of $\Lambda_{\mathcal N}$ and
$\mathcal A_{\mathcal N}$ on $V_i$.  Applying
Theorem~\ref{thm:affine-monodromy} to the corresponding quotient
isogeny gives
\[
 G_{f_i}\cong V_i\rtimes\Lambda_i,
 \qquad
 A_{f_i}\cong V_i\rtimes\mathcal A_i.
\]
\end{remark}

Recall that a coset $\xi\in\mathcal C_{\mathcal N}$ belongs to
$\mathcal C_{\mathcal N}^{\mathrm{fpf}}$ exactly when
$\ker(a-1)=\{0\}$ for every $a\in\xi$; equivalently, every
element of $\xi$ acts without fixed points on
$\mathcal N\setminus\{0\}$.

\begin{lemma}[Characterization by Frobenius cosets]
\label{lem:frobenius-coset-characterization}
For every $r\geqslant1$,
\[
 r\in\mathcal D_{\mathrm{perm}}(f)
 \Longleftrightarrow
 \overline{\mathcal F}_{\mathcal N}^{\,r}
 \in\mathcal C_{\mathcal N}^{\mathrm{fpf}}.
\]
\end{lemma}

\begin{proof}
By Corollary~\ref{cor:extension-kernel},
$r\in\mathcal D_{\mathrm{perm}}(f)$ exactly when
$\ker(\mathcal F_{\mathcal N}^r-u)=\{0\}$ for every
$u\in\Lambda_{\mathcal N}$.  A nonzero $D\in\mathcal N$ satisfies
$\mathcal F_{\mathcal N}^r(D)=u(D)$ if and only if
$u^{-1}\mathcal F_{\mathcal N}^r$ fixes $D$.  As $u$ ranges over
$\Lambda_{\mathcal N}$, so does $u^{-1}$; hence the kernel condition is
exactly that no element of the coset
$\Lambda_{\mathcal N}\mathcal F_{\mathcal N}^r
 =\mathcal F_{\mathcal N}^r\Lambda_{\mathcal N}$ fixes a nonzero point of
$\mathcal N$.  This is precisely the defining condition for membership in
$\mathcal C_{\mathcal N}^{\mathrm{fpf}}$.
\end{proof}

\begin{proof}[Proof of Theorem~\ref{thm:intro-permutation-extension-criterion}]
Under the isomorphism
\[
 \mathbb Z/c_{\mathcal N}\mathbb Z\xrightarrow{\sim}
 \mathcal C_{\mathcal N},
 \qquad j\longmapsto\overline{\mathcal F}_{\mathcal N}^{\,j},
\]
Lemma~\ref{lem:frobenius-coset-characterization} identifies
$\mathcal D_{\mathrm{perm}}(f)$ with the positive integers whose residue
classes correspond to $\mathcal C_{\mathcal N}^{\mathrm{fpf}}$.  Hence the
set is purely periodic and has limiting proportion
$\lvert\mathcal C_{\mathcal N}^{\mathrm{fpf}}\rvert/c_{\mathcal N}$.
A positive integer $\nu$ is a period exactly when left translation by
$\overline{\mathcal F}_{\mathcal N}^{\,\nu}$ preserves
$\mathcal C_{\mathcal N}^{\mathrm{fpf}}$, equivalently when
$\overline{\mathcal F}_{\mathcal N}^{\,\nu}\in H_{\mathcal N}$.  Since
$\mathcal C_{\mathcal N}$ is cyclic and generated by
$\overline{\mathcal F}_{\mathcal N}$, the least such positive $\nu$ is
$[\mathcal C_{\mathcal N}:H_{\mathcal N}]$.  The periodic set is infinite
exactly when $\mathcal C_{\mathcal N}^{\mathrm{fpf}}$ is nonempty.  If
$\mathcal N\ne0$, the identity element of $\mathcal A_{\mathcal N}$ fixes
all nonzero points of $\mathcal N$, so the identity coset does not belong to
$\mathcal C_{\mathcal N}^{\mathrm{fpf}}$.  Thus
$c_{\mathcal N}\mid r$ implies $r\notin\mathcal D_{\mathrm{perm}}(f)$.
\end{proof}

\begin{corollary}[Permutation--exceptionality equivalence]
\label{cor:permutation-exceptional}
For every $r\geqslant1$,
\[
 f\text{ permutes }\mathbf P^1(k_r)
 \Longleftrightarrow
 f\text{ is exceptional over }k_r.
\]
In particular, the following are equivalent:
\begin{enumerate}
\item[\textup{(i)}] $f$ permutes $\mathbf P^1(k)$;
\item[\textup{(ii)}] $f$ permutes $\mathbf P^1(k_r)$ for at least one $r\geqslant1$;
\item[\textup{(iii)}] $f$ is exceptional over $k$.
\end{enumerate}
\end{corollary}

\begin{proof}
After base change to $k_r$, the same equivariant-isogeny datum has Frobenius
$\mathcal F_E^r$, so Theorem~\ref{thm:intro-permutation-extension-criterion}
applies over $k_r$.  If $f$ permutes $\mathbf P^1(k_r)$, then degree $1$
belongs to the corresponding permutation extension-degree set; the theorem
therefore makes that set infinite, so $f/k_r$ is exceptional.  Conversely,
a collision between two $k_r$-rational points persists over every finite
extension of $k_r$, so exceptionality forces injectivity, hence bijectivity,
on $\mathbf P^1(k_r)$.  Finally, if $f$ permutes over some $k_r$, then
$\mathcal D_{\mathrm{perm}}(f)$ is nonempty, so Theorem~\ref{thm:intro-permutation-extension-criterion}
makes $f$ exceptional over $k$; the case $r=1$ gives the converse implication
needed for the final three-way equivalence.
\end{proof}

Together with Corollary~\ref{cor:extension-kernel}, this proves
Theorem~\ref{thm:intro-kernel}.  Combined with
Theorem~\ref{thm:intro-reconstruction}, it also proves
Corollary~\ref{cor:intro-exhaustive}.

\begin{proof}[Proof of Theorem~\ref{thm:intro-genus-at-most-one-permutation-exceptionality}]
Fix $r\geqslant1$ and put $K=k_r$.  The field $\bar k$ is also an
algebraic closure of $K$, and extending the constants does not change the
genus of the Galois closure.  Suppose first that $f$ permutes
$\mathbf P^1(K)$, and choose a complete $K$-decomposition
$f=f_s\circ\cdots\circ f_1$, with each $f_i$ $K$-indecomposable.  Since $f$
is separable, every $f_i$ is separable, and bijectivity of the composite on the
finite set $\mathbf P^1(K)$ forces every $f_i$ to be bijective there.

Use the fixed geometric Galois closure
$\overline{\Omega}_f/\bar k(f(\mathbf x))$, and put
$C:=(C_f)_{\bar k}$.  Set
\[
 \begin{aligned}
 \mathbf x_0&:=\mathbf x,
 &\qquad \mathbf x_i&:=f_i(\mathbf x_{i-1})\quad(1\leqslant i\leqslant s),\\
 J_i&:=\operatorname{Gal}(\overline{\Omega}_f/\bar k(\mathbf x_i)),\\
 \mathcal K_i&:=\bigcap_{a\in J_i}aJ_{i-1}a^{-1}.
 \end{aligned}
\]
Then $J_0<J_1<\cdots<J_s$.  For each $i$, $\mathcal K_i$ is
the largest normal subgroup of $J_i$ contained in $J_{i-1}$.  Indeed,
$\overline{\Omega}_f/\bar k(\mathbf x_i)$ is Galois with group $J_i$ and
$\bar k(\mathbf x_{i-1})=\overline{\Omega}_f^{J_{i-1}}$, so the normal
closure of $\bar k(\mathbf x_{i-1})/\bar k(\mathbf x_i)$ inside
$\overline{\Omega}_f$ is $\overline{\Omega}_f^{\mathcal K_i}$.  Thus,
under the coordinate identification $X\mapsto\mathbf x_{i-1}$, this fixed
field realizes the geometric Galois closure of the adjacent factor $f_i$, with
Galois-closure curve $C/\mathcal K_i$.  Riemann--Hurwitz therefore gives genus
at most one.  If that genus is zero, the
genus-zero permutation--exceptionality theorem
\cite[Theorem~9.8]{FanG0} makes $f_i/K$ exceptional; if it is one,
Corollary~\ref{cor:intro-exhaustive} does the same.

Each $f_i$ is separable, so over the finite field $K$ arithmetic
exceptionality is equivalent to exceptionality of its monodromy triple
\cite[\S2.8]{GMS03}.  Hence the faithful monodromy triple of every $f_i$ is
exceptional.  For a two-step composition, choose an arithmetic Galois closure
and write
\[
 M<U<A,
 \qquad
 G\triangleleft A,
\]
for the point stabilizer, the intermediate subgroup, and the arithmetic and
geometric groups.  Geometric monodromy is transitive on the sheets, so
$A=GM$; moreover $A/G$ is cyclic over the finite field $K$.  Thus the
hypotheses of \cite[Lemma~3.5]{GMS03} apply.  Its two induced coset actions are
\[
 (A,G,A/U)
 \qquad\text{and}\qquad
 (U,G\cap U,U/M).
\]
Their permutation kernels are respectively
$\operatorname{core}_A(U)$ and $\operatorname{core}_U(M)$; after quotienting
by these kernels, the arithmetic and geometric images are precisely the
monodromy triples of the outer and inner factors.  Exceptionality depends only
on these permutation images, so both induced triples are exceptional.
Lemma~3.5 therefore makes the two-step composite triple exceptional, and
\cite[\S2.8]{GMS03} converts this back to exceptionality of the composite
over $K$.  Iterating the two-step argument proves that the full composite
$f/K$ is exceptional.

Conversely, if $f/K$ is exceptional, it permutes $\mathbf P^1$ over infinitely
many finite extensions of $K$.  A collision between two distinct $K$-rational
points would persist over every such extension, so no collision exists.  Hence
$f$ is injective, and therefore bijective, on the finite set
$\mathbf P^1(K)$.  This proves the theorem.
\end{proof}

\begin{proof}[Proof of Theorem~\ref{thm:intro-genus-at-most-one-decomposition-extension-degrees}]
Let $f=f_s\circ\cdots\circ f_1$ be a $k$-decomposition.  Since $f$ is
separable, every factor $f_i$ is separable.  Use the fixed geometric Galois
closure $\overline{\Omega}_f/\bar k(f(\mathbf x))$ and put
$C:=(C_f)_{\bar k}$.  Set
\[
 \begin{aligned}
 \mathbf x_0&:=\mathbf x,
 &\qquad \mathbf x_i&:=f_i(\mathbf x_{i-1})\quad(1\leqslant i\leqslant s),\\
 J_i&:=\operatorname{Gal}(\overline{\Omega}_f/\bar k(\mathbf x_i)),\\
 \mathcal K_i&:=\bigcap_{a\in J_i}aJ_{i-1}a^{-1}.
 \end{aligned}
\]
Then $J_0<J_1<\cdots<J_s$.  For each $i$, $\mathcal K_i$ is
the largest normal subgroup of $J_i$ contained in $J_{i-1}$.  As above,
$\overline{\Omega}_f/\bar k(\mathbf x_i)$ is Galois with group $J_i$ and
$\bar k(\mathbf x_{i-1})=\overline{\Omega}_f^{J_{i-1}}$, so
$\overline{\Omega}_f^{\mathcal K_i}$ is the normal closure of
$\bar k(\mathbf x_{i-1})/\bar k(\mathbf x_i)$ inside
$\overline{\Omega}_f$.  Under the coordinate identification
$X\mapsto\mathbf x_{i-1}$, it therefore realizes the geometric Galois
closure of $f_i$, with curve $C/\mathcal K_i$.  Riemann--Hurwitz shows that
the Galois closure of every $f_i$ has genus at most one.

For each $r\geqslant1$, a composite of self-maps of the finite set
$\mathbf P^1(k_r)$ is bijective if and only if every factor is bijective.  Hence
\[
 \mathcal D_{\mathrm{perm}}(f)
 =\bigcap_{i=1}^s\mathcal D_{\mathrm{perm}}(f_i).
\]
Applying
Theorem~\ref{thm:intro-genus-at-most-one-permutation-exceptionality} to $f$ and
to every $f_i$ identifies each permutation extension-degree set with the
corresponding exceptional extension-degree set, so the preceding equality
gives the displayed formula for $\mathcal D(f)$.  The final assertion follows
immediately.
\end{proof}

\begin{proposition}[Rank-one kernels]
Assume $\mathcal N\cong C_\ell$, where $\ell$ is prime, and write
$\mathcal F_{\mathcal N}(P)=\lambda P$ with
$\lambda\in\mathbb F_\ell^\times$.  Then
\begin{enumerate}
\item[(1)] $\mathcal D(f)
=\{r\geqslant1:c_{\mathcal N}\nmid r\}$, where
\[
 c_{\mathcal N}
 =\operatorname{ord}_{\mathbb F_\ell^\times/\Lambda_{\mathcal N}}
 (\lambda\Lambda_{\mathcal N});
\]
\item[(2)] $f$ is exceptional if and only if $c_{\mathcal N}>1$;
\item[(3)] if $f$ is exceptional, then
$\operatorname{per}(f)=c_{\mathcal N}$ and
$\delta_f=1-1/c_{\mathcal N}$.
\end{enumerate}
If, in addition, $\Lambda_{\mathcal N}=\langle\rho\rangle
\leq\mathbb F_\ell^\times$ has order $m$, then
\[
 c_{\mathcal N}
 =\operatorname{ord}_{\mathbb F_\ell^\times}(\lambda^m),
\]
and consequently
\[
 \mathcal D(f)
 =\{r\geqslant1:
 \operatorname{ord}(\lambda^m)\nmid r\}.
\]
\end{proposition}

\begin{proof}
Because the isogeny is separable, $\mathcal N$ is a reduced geometric
kernel; this includes the possible rank-one case in characteristic $\ell$.
Identify the abstract group $\mathcal N$ with the one-dimensional
$\mathbb F_\ell$-space $\mathbb F_\ell$; every automorphism then acts by a
scalar.  A scalar fixes no nonzero point exactly when it is not $1$.
Thus the Frobenius coset belongs to $\mathcal C_{\mathcal N}^{\mathrm{fpf}}$ exactly when
$\lambda^r\notin\Lambda_{\mathcal N}$, which proves the first three
assertions.  In the cyclic case, $\langle\rho\rangle$ is the subgroup of
$m$th roots of unity in $\mathbb F_\ell^\times$.  Hence the order of the
coset $\lambda\langle\rho\rangle$ is exactly the order of $\lambda^m$,
which gives the final formulas.
\end{proof}

\begin{corollary}[Finite determinant test]
\label{cor:finite-determinant-test}
Assume $\varphi=[n]$ with $\gcd(n,q)=1$.  For each integer $j$ with
$0\leqslant j<c_{\mathcal N}$, put
\[
 D_j
 :=\prod_{u\in\Lambda_{\mathcal N}}
 \det\!\left(
 \mathcal F_{\mathcal N}^{\,j}-u
 \mid E[n]
 \right).
\]
Then $r\in\mathcal D(f)$ if and only if
$D_j\in(\mathbb Z/n\mathbb Z)^\times$, where
$j\in\{0,\ldots,c_{\mathcal N}-1\}$ is the least nonnegative residue of
$r$ modulo $c_{\mathcal N}$.
\end{corollary}

\begin{proof}
A coset belongs to $\mathcal C_{\mathcal N}^{\mathrm{fpf}}$ exactly when every displayed endomorphism is invertible
on $E[n]$, equivalently when all determinants are units.
\end{proof}

\begin{remark}[Finite procedure]
\label{rem:finite-procedure}
Theorem~\ref{thm:intro-permutation-extension-criterion} yields an explicit finite procedure for a
prescribed equivariant isogeny datum over $k$.  Put
\(L=\Lambda_{\mathcal N}, \qquad F=\mathcal F_{\mathcal N}, \qquad c=\min\{\nu\geqslant1:F^\nu\in L\}.\)
For each integer $j$ with $0\leqslant j<c$, test whether
$\ker(F^j-u)=\{0\}$ for every $u\in L$, and let $A\subseteq\mathbb Z/c\mathbb Z$ be the set of residue classes
represented by those $j$.  Then
\[
 \mathcal D(f)
 =\{r\geqslant1:r\bmod c\in A\},
 \qquad
 \delta_f=\frac{\lvert A\rvert}{c},
\]
while $\operatorname{per}(f)$ is the least positive integer $\nu$ such that
\(A+\nu=A\)
in $\mathbb Z/c\mathbb Z$.  In the multiplication case,
Corollary~\ref{cor:finite-determinant-test} replaces these kernel tests by the
determinant unit tests $D_j\in(\mathbb Z/n\mathbb Z)^\times$.  Thus the procedure decides the extension-degree arithmetic for a prescribed
equivariant isogeny datum; finding suitable kernels or isogenies is a separate
algorithmic problem.
\end{remark}

\section{Cyclic symmetries and explicit extension-degree criteria}
\label{sec:cyclic}

This section applies the finite-kernel criterion to cyclic linear symmetries,
first deriving ground-field criteria and linear models and then giving explicit
rank-two formulas.

\subsection{Prime-to-characteristic cyclic Frobenius}

\begin{lemma}[Prime-to-characteristic cyclic Frobenius]
\label{lem:cyclic-frobenius-prime-to-p}
Let $L=\langle\beta\rangle\leq\operatorname{Aut}_{\bar k}(E,\mathrm O_E)$
be a cyclic group of order $m$, stable under Frobenius conjugation, with
$\operatorname{char}k\nmid m$.  Then
\begin{equation}
\label{eq:cyclotomic}
 \mathcal F_E\beta\mathcal F_E^{-1}=\beta^q.
\end{equation}
Equivalently, the induced Frobenius automorphism of $L$ is
$\vartheta_L(\beta)=\beta^q$.
\end{lemma}

\begin{proof}
The tangent character of $L$ at the origin is faithful.  Indeed, if a
nontrivial element $\gamma$ of order $n$ prime to the characteristic had
derivative $1$, then in a local parameter $t$ one could write
\(\gamma(t)=t+c_st^s+O(t^{s+1}), \qquad c_s\ne0, \qquad s>1.\)
Iteration gives
\(\gamma^n(t)=t+n c_st^s+O(t^{s+1}),\)
contradicting $\gamma^n=1$ because $n$ is invertible in $k$.  Thus
$\beta$ acts on the tangent space by a primitive $m$th root of unity.
Frobenius conjugation raises this tangent eigenvalue to its $q$th power,
and faithfulness of the tangent character gives \eqref{eq:cyclotomic}.
\end{proof}

We call the Frobenius action on $L$ \emph{split} if it is trivial on $L$.
When $m>2$, we call it \emph{nonsplit} if it acts by inversion on $L$.
By \eqref{eq:cyclotomic}, these conditions are equivalent to
$q\equiv1\pmod m$ and $q\equiv-1\pmod m$, respectively.  For $m=2$,
inversion is the identity, and we use only the term split.

\subsection{Good-characteristic cyclic criteria and linear models}

For the explicit ground-field criteria and linear models in this subsection,
assume $\operatorname{char}k>3$ and put
\(L=L_E=\operatorname{im} (\Lambda\to\operatorname{Aut}_{\bar k}(E,\mathrm O_E)).\)
By \cite[Theorem~III.10.1 and Corollary~III.10.2]{Sil09}, the nontrivial
group $L$ is cyclic; write
\(L=\langle\beta\rangle\cong C_m, \qquad m\in\{2,3,4,6\}.\)
Moreover, $j(E)=0$ for $m=3,6$ and $j(E)=1728$ for $m=4$.
Lemma~\ref{lem:cyclic-frobenius-prime-to-p} gives
\eqref{eq:cyclotomic}.  For $m\in\{2,3,4,6\}$, the split and nonsplit
cases defined above exhaust the possibilities, with only the split case for
$m=2$.

The kernel criterion depends only on $L$: requiring
$\mathcal N\cap\ker(\mathcal F_E-u_{\gamma,E})=\{\mathrm O_E\}$ for every
$\gamma\in\Lambda$ is equivalent to requiring the same condition
for every $u\in L$.  Moreover this condition is constant on the $\sim_{\vartheta_L}$-classes in $L$.

\begin{lemma}[Equivalence classes in a cyclic group]
\label{lem:cyclic-classes}
Suppose $\vartheta_L(\beta)=\beta^a$ with
$a\in(\mathbb Z/m\mathbb Z)^\times$.  Then
\[
 \beta^j
 \sim_{\vartheta_L}
 \beta^{j+(a-1)t}
 \qquad\text{for all }j,t\in\mathbb Z.
\]
Hence the $\sim_{\vartheta_L}$-classes are indexed by
\[
 \mathbb Z/m\mathbb Z
 \Big/
 \bigl((a-1)(\mathbb Z/m\mathbb Z)\bigr).
\]
There are $d=\gcd(m,a-1)$ classes, represented by
$1,\beta,\ldots,\beta^{d-1}$.
\end{lemma}

\begin{proof}
Since $L$ is abelian,
\(\vartheta_L(\beta^t)\beta^j\beta^{-t} =\beta^{at+j-t} =\beta^{j+(a-1)t}.\)
The remaining assertions follow.
\end{proof}

\begin{theorem}[The four cyclic criteria]
\label{thm:cyclic}
Retain the standing assumptions of this subsection, in particular
$\operatorname{char}k>3$, and let $\mathcal N=\ker\varphi(\bar k)$.  Then
$f$ permutes $\mathbf P^1(k)$ if and only if the corresponding condition
holds for the cyclic group $L$ of linear parts:
\begin{enumerate}
\item[(1)] if $L\cong C_2$, then
$\mathcal N\cap E(\mathbb F_{q^2})=\{\mathrm O_E\}$;
\item[(2)] if $L\cong C_3$, then
$\mathcal N\cap E(\mathbb F_{q^3})=\{\mathrm O_E\}$ for
$q\equiv1\pmod3$, while
$\mathcal N\cap E(k)=\{\mathrm O_E\}$ for $q\equiv2\pmod3$;
\item[(3)] if $L\cong C_4$, then
$\mathcal N\cap E(\mathbb F_{q^4})=\{\mathrm O_E\}$ for
$q\equiv1\pmod4$, while
$\mathcal N\cap E(\mathbb F_{q^2})=\{\mathrm O_E\}$ for
$q\equiv3\pmod4$;
\item[(4)] if $L\cong C_6$, then
$\mathcal N\cap E(\mathbb F_{q^6})=\{\mathrm O_E\}$ for
$q\equiv1\pmod6$, while
$\mathcal N\cap E(\mathbb F_{q^2})=\{\mathrm O_E\}$ for
$q\equiv5\pmod6$.
\end{enumerate}
\end{theorem}

\begin{proof}
Suppose first that Frobenius acts trivially on $L$.  Then it commutes with
$\beta$, and on $\mathcal N$
\[
 \mathcal F_E^m-1
 =\prod_{j=0}^{m-1}(\mathcal F_E-\beta^j).
\]
The commuting factors are all injective exactly when their product is, so
Theorem~\ref{thm:kernel} gives
\[
 f\text{ permutes }\mathbf P^1(k)
 \Longleftrightarrow
 \mathcal N\cap E(\mathbb F_{q^m})=\{\mathrm O_E\}.
\]
This gives all split cases.

For nonsplit $C_3$, Lemma~\ref{lem:cyclic-classes} gives a single
$\sim_{\vartheta_L}$-class in $L$, represented by $1$, so the
criterion reduces to
$\mathcal N\cap\ker(\mathcal F_E-1)=\{\mathrm O_E\}$, equivalently
$\mathcal N\cap E(k)=\{\mathrm O_E\}$.  Finally, for nonsplit $C_4$ or $C_6$, if
$\mathcal F_E(P)=\beta^jP$ then
$\mathcal F_E^2(P)=P$.  Conversely, if a nonzero $P\in\mathcal N$ satisfies
$\mathcal F_E^2(P)=P$, then either $(\mathcal F_E+1)P=0$, or the nonzero point
$(\mathcal F_E+1)P$ lies in $\ker(\mathcal F_E-1)$.  Since $\pm1\in L$,
the kernel criterion fails in either case.  Hence the criterion is
$\mathcal N\cap E(\mathbb F_{q^2})=\{\mathrm O_E\}$.
\end{proof}

\begin{remark}[Nonsplit $C_3$ quotient on rational points]
In the origin-preserving special case $\Lambda=L\cong C_3$ with
$q\equiv2\pmod3$, the quotient map itself induces
$E(k)\xrightarrow{\sim}(E/C_3)(k)$.  Indeed the unique $\sim_{\vartheta_L}$-class has size $3$, so Proposition~\ref{prop:weighted-fibre-counting} gives $3m_Q(1)=3$.
\end{remark}

After a $\bar k$-isomorphism, the linear pairs $(E,L)$ admit the following
normal forms; the listed coordinates generate the fixed fields of the
origin-preserving linear actions:
\begin{equation}
\label{eq:models}
\begin{array}{c|c|c}
L&E&\text{linear quotient coordinate}\\ \hline
C_2&y^2=x^3+ax+b,\ \Delta_E\ne0,\ \beta(x,y)=(x,-y)&x\\
C_3&y^2=x^3+b,\ b\ne0,\ \beta(x,y)=(\zeta_3x,y)&y\\
C_4&y^2=x^3+ax,\ a\ne0,\ \beta(x,y)=(-x,\zeta_4y)&x^2\\
C_6&y^2=x^3+b,\ b\ne0,\ \beta(x,y)=(\zeta_3x,-y)&x^3.
\end{array}
\end{equation}
Here $\zeta_3^2+\zeta_3+1=0$ and $\zeta_4^2=-1$.  These explicit linear
actions also make \eqref{eq:cyclotomic} transparent.  Changing the translation component can alter the $k$-rational quotient
morphism and the chosen quotient coordinate, so
\eqref{eq:models} should be read as a table of linear models.

\begin{corollary}[Coordinate expression of the quotient map]
Let $(E,E',\Lambda,\varphi)$ be an equivariant isogeny datum over $k$.
Choose $k$-coordinates $z$ on $E/\Lambda$ and $z'$ on $E'/\Lambda$, and
denote their pullbacks to $E$ and $E'$ by the same symbols.  Then the lower
rational function is the unique $f\in k(X)$ satisfying
\[
 z'\circ\varphi=f\circ z.
\]
In the origin-preserving linear cases of \eqref{eq:models}, one may take
$z=x,y,x^2,x^3$ for $L=C_2,C_3,C_4,C_6$, respectively, and analogously on
the target.  Thus the coordinate formulas for $\varphi$ determine the
corresponding $f(X)$ explicitly.
\end{corollary}

\begin{proof}
The function $z'\circ\varphi$ is $\Lambda$-invariant by equivariance, hence
belongs to the quotient function field $k(E/\Lambda)=k(z)$.  Thus there is a
unique $f\in k(X)$ with $z'\circ\varphi=f\circ z$, and this is precisely the
coordinate expression of the lower morphism.  The final assertion follows
from the invariant generators listed in \eqref{eq:models}.
\end{proof}

Example~\ref{ex:c4} carries out this computation for a $C_4$ quotient, while
Example~\ref{ex:affine-c2} exhibits a genuinely affine quotient coordinate.

\begin{proposition}[Point-count criteria for multiplication maps]
\label{prop:point-count-multiplication-maps}
Retain the standing assumptions of this subsection, in particular
$\operatorname{char}k>3$.  Let $\varphi=[n]$ with $\gcd(n,q)=1$, and put
$a_q=q+1-\lvert E(k)\rvert$.  Then the criteria for $f$ to permute
$\mathbf P^1(k)$ are:
\begin{enumerate}
\item[(1)] for $L=C_2$, $\gcd(n,(q+1)^2-a_q^2)=1$;
\item[(2)] for split $C_3$, $\gcd(n,q^3+1-a_q^3+3qa_q)=1$;
\item[(3)] for nonsplit $C_3$, $\gcd(n,q+1)=1$;
\item[(4)] for split $C_4$, $\gcd(n,q^4+1-a_q^4+4qa_q^2-2q^2)=1$;
\item[(5)] for nonsplit $C_4$, $\gcd(n,(q+1)^2-a_q^2)=1$;
\item[(6)] for split $C_6$,
$\gcd(n,q^6+1-a_q^6+6qa_q^4-9q^2a_q^2+2q^3)=1$;
\item[(7)] for nonsplit $C_6$, $\gcd(n,(q+1)^2-a_q^2)=1$.
\end{enumerate}
\end{proposition}

\begin{proof}
For $\gcd(n,q)=1$,
\(E[n]\cap E(\mathbb F_{q^s})=\{\mathrm O_E\} \Longleftrightarrow \gcd(n,\lvert E(\mathbb F_{q^s})\rvert)=1.\)
If $\alpha,\bar\alpha$ are the Frobenius eigenvalues and
$a_{q^s}=\alpha^s+\bar\alpha^s$, then
\[
 \lvert E(\mathbb F_{q^s})\rvert=q^s+1-a_{q^s},
 \qquad
 a_{q^0}=2,
 \quad a_{q^1}=a_q,
 \quad
 a_{q^{s+1}}=a_q a_{q^s}-q a_{q^{s-1}},
\]
by the standard Frobenius point-count formula
\cite[Chap.~V, \S2]{Sil09}.  This gives the displayed expressions.  For nonsplit $C_3$,
$j(E)=0$ and $q\equiv2\pmod3$, so $E$ admits a $k$-model
\(y^2=x^3+b, \qquad b\in k^\times.\)
Since $\gcd(3,q-1)=1$, the cube map $x\mapsto x^3$ is a bijection of $k$.
Let $\chi$ be the quadratic character of $k$, extended by $\chi(0)=0$.
Then
\(\sum_{x\in k}\chi(x^3+b) =\sum_{z\in k}\chi(z+b) =\sum_{w\in k}\chi(w) =0.\)
Hence $\lvert E(k)\rvert=q+1$, which gives the nonsplit $C_3$ criterion.
\end{proof}

Item~\textup{(1)} is the classical $C_2$ Latt\`es permutation criterion.
For multiplication maps it is exactly \cite[Corollary~2.6]{Bell22}.  In the
CM endomorphism setting the corresponding ideal-theoretic criterion is
\cite[Corollary~2.8]{Kuc14}; Bell et al. explicitly note the essential
equivalence of the two formulations \cite[Remark~2.7]{Bell22}.  In the
origin-preserving prime-degree $C_2$ case, Theorem~\ref{thm:cyclic} also
contains \cite[Theorem~1]{BT18} as a special case.

We include the $C_2$ row for comparison with the classical criterion.  The purpose of Theorem~\ref{thm:cyclic} and the table above is not to
introduce new cyclic quotient geometries, but to place the $C_2$, $C_3$,
$C_4$, and $C_6$ symmetries under one finite-field kernel criterion.  The identical numerical expression in the nonsplit $C_4$ and $C_6$ rows
comes from the same inversion action of Frobenius, although the quotient
symmetries are different.
The argument for a prescribed affine equivariant-isogeny datum itself does
not require prime degree.

\begin{corollary}[Cyclic criteria after base extension]
Retain the standing assumptions of this subsection, in particular
$\operatorname{char}k>3$.  Then $r\in\mathcal D(f)$ if and only if the
corresponding condition below holds:
\begin{enumerate}
\item[(1)] if $L\cong C_2$, then
$\mathcal N\cap E(\mathbb F_{q^{2r}})=\{\mathrm O_E\}$;
\item[(2)] if $L\cong C_3$ and $q\equiv1\pmod3$, then
$\mathcal N\cap E(\mathbb F_{q^{3r}})=\{\mathrm O_E\}$;
\item[(3)] if $L\cong C_3$ and $q\equiv2\pmod3$, then
\[
 \begin{cases}
 \mathcal N\cap E(\mathbb F_{q^r})=\{\mathrm O_E\},&r\text{ odd},\\
 \mathcal N\cap E(\mathbb F_{q^{3r}})=\{\mathrm O_E\},&r\text{ even};
 \end{cases}
\]
\item[(4)] if $L\cong C_4$ and $q\equiv1\pmod4$, then
$\mathcal N\cap E(\mathbb F_{q^{4r}})=\{\mathrm O_E\}$;
\item[(5)] if $L\cong C_4$ and $q\equiv3\pmod4$, then
\[
 \begin{cases}
 \mathcal N\cap E(\mathbb F_{q^{2r}})=\{\mathrm O_E\},&r\text{ odd},\\
 \mathcal N\cap E(\mathbb F_{q^{4r}})=\{\mathrm O_E\},&r\text{ even};
 \end{cases}
\]
\item[(6)] if $L\cong C_6$ and $q\equiv1\pmod6$, then
$\mathcal N\cap E(\mathbb F_{q^{6r}})=\{\mathrm O_E\}$;
\item[(7)] if $L\cong C_6$ and $q\equiv5\pmod6$, then
\[
 \begin{cases}
 \mathcal N\cap E(\mathbb F_{q^{2r}})=\{\mathrm O_E\},&r\text{ odd},\\
 \mathcal N\cap E(\mathbb F_{q^{6r}})=\{\mathrm O_E\},&r\text{ even}.
 \end{cases}
\]
\end{enumerate}
\end{corollary}

\begin{proof}
Over $k_r=\mathbb F_{q^r}$ the group $L$ of linear parts is unchanged, while
Frobenius acts on its generator by $\beta\mapsto\beta^{q^r}$.
Apply Theorem~\ref{thm:cyclic} over $k_r$.  If the original action is split,
it remains split for every $r$.  If it is nonsplit, then $q^r\equiv q$ modulo
$m$ for odd $r$ and $q^r\equiv1\pmod m$ for even $r$, giving the displayed
alternatives.
\end{proof}

\subsection{Explicit rank-two formulas}
\label{subsec:explicit-rank-two-formulas}

Assume $\mathcal N=E[\ell](\bar k)$, where $\ell$ is odd,
$\ell\ne\operatorname{char}k$, and $\ell\nmid m$.  Let
\(L=\langle\beta\rangle\cong C_m,\quad m\in\{2,3,4,6\},\quad
\operatorname{char}k\nmid m,\)
be stable under Frobenius conjugation, and suppose its restriction to
$E[\ell]$ has exact order $m$ and equals
\(\Lambda_{\mathcal N}=\langle\beta|_{E[\ell]}\rangle.\)
Write $\beta$ for this restriction and put $F=\mathcal F_E|_{E[\ell]}$.
Then $F\beta F^{-1}=\beta^q$ by
Lemma~\ref{lem:cyclic-frobenius-prime-to-p}.  We use the split/nonsplit
convention above.  No exceptionality hypothesis is imposed; the formulas determine exactly the
extension degrees for which the map is a permutation, equivalently exceptional.  These hypotheses also cover the tame cases $m=3$ in characteristic $2$ and
$m=2,4$ in characteristic $3$.

Put \(a_q:=q+1-\lvert E(k)\rvert.\)  The characteristic polynomial of
$F$ on $E[\ell]$ is the reduction modulo $\ell$ of
\(T^2-a_qT+q\); let
$\alpha_1,\alpha_2\in\overline{\mathbb F}_\ell^\times$ be its roots.

\begin{lemma}[Determinant and eigenvalue identities]
\label{lem:determinant-eigenvalue-identities}
On $E[\ell]$ one has
\[
 \det(\beta)=1,
 \qquad
 \det(F)=q
 \quad\text{in }\mathbb F_\ell^\times,
\]
where $q$ on the right denotes its residue class modulo $\ell$.
If $q\equiv1\pmod m$, then for every $a\in\mathbb Z$ the operator
$\beta^aF$ has the same $m$th powers of eigenvalues as $F$.  If $m>2$ and
$q\equiv-1\pmod m$, then $(\beta^aF)^2=F^2$.
\end{lemma}

\begin{proof}
For $m=2$, $\beta=-I$, so $\det(\beta)=1$.  Let $m>2$.  Since
$\ell\nmid m$, the operator $\beta$ is semisimple.  As it is the
restriction of a degree-one elliptic-curve automorphism,
\cite[Proposition~III.8.6]{Sil09} gives $\det(\beta)=1$ on $E[\ell]$;
hence its eigenvalues are $\rho,\rho^{-1}$, both of exact order $m$.
The same proposition gives $\det(F)=q$ in $\mathbb F_\ell^\times$.
In the split case $F$ commutes with $\beta$, so multiplying by $\beta^a$
changes the two Frobenius eigenvalues by $m$th roots of unity.  In the
nonsplit case $F\beta^a=\beta^{-a}F$, whence $(\beta^aF)^2=F^2$.
\end{proof}

\subsubsection{Split rank-two formulas}

Assume $q\equiv1\pmod m$; for $m=2$ this is automatic.  Put \(d_i:=\operatorname{ord}_{\overline{\mathbb F}_\ell^\times}(\alpha_i^m), \qquad i=1,2\).

\begin{theorem}[Split spectral criterion]
\label{thm:split-spectral-criterion}
Under the standing hypotheses of this subsection, assume $q\equiv1\pmod m$.
For every $r\geqslant1$,
\begin{align}
 r\in\mathcal D(f)
 &\Longleftrightarrow
 \det(\mathcal F_E^{mr}-I\mid E[\ell])\ne0
 \label{eq:split-rank-two-det}\\
 &\Longleftrightarrow
 d_1\nmid r\ \text{and}\ d_2\nmid r
 \label{eq:split-rank-two-order}\\
 &\Longleftrightarrow
 \ell\nmid\lvert E(\mathbb F_{q^{mr}})\rvert.
 \label{eq:split-rank-two-pointcount}
\end{align}
Moreover,
\[
 \begin{aligned}
 \delta_f&=
 1-\frac1{d_1}-\frac1{d_2}
 +\frac1{\operatorname{lcm}(d_1,d_2)},\\
 \operatorname{per}(f)&=
 \begin{cases}
 d_1,&d_1\mid d_2,\\
 d_2,&d_2\mid d_1,\\
 \operatorname{lcm}(d_1,d_2),&\text{otherwise}.
 \end{cases}
 \end{aligned}
\]
\end{theorem}

\begin{proof}
Assume first that $m>2$.  By
Lemma~\ref{lem:determinant-eigenvalue-identities}, $\beta$ has distinct
eigenvalues $\rho,\rho^{-1}$ of exact order $m$.  Since $F$ commutes with
$\beta$, the two $\beta$-eigenlines are $F$-stable; label the corresponding
$F$-eigenvalues by $\alpha_1,\alpha_2$.  On either line,
$F^r-\beta^j$ is singular for some $j$ exactly when
$\alpha_i^r\in\langle\rho\rangle$, equivalently
$\alpha_i^{mr}=1$.  The extension-field kernel criterion therefore gives
\eqref{eq:split-rank-two-det} and \eqref{eq:split-rank-two-order}.
For $m=2$, the forbidden operators are $F^r\mp I$; one is singular exactly
when some $\alpha_i^{2r}=1$, giving the same conclusion.

Finally,
\[
 \det(I-\mathcal F_E^s\mid E[\ell])
 =\lvert E(\mathbb F_{q^s})\rvert
 \quad\text{in }\mathbb F_\ell,
\]
with the integer reduced modulo $\ell$; taking $s=mr$ gives
\eqref{eq:split-rank-two-pointcount}.  The excluded degrees are the union of the multiples of $d_1$ and $d_2$,
so inclusion--exclusion gives $\delta_f$.  If one $d_i$ divides the other,
the least period is the smaller one.  Otherwise let
$L=\operatorname{lcm}(d_1,d_2)$ and let $H_i\leq\mathbb Z/L\mathbb Z$ be
the multiples of $d_i$.  The translation stabilizer of $H_1\cup H_2$ is
trivial: if $0\ne t\in H_1\setminus H_2$, choose
$h\in H_2\setminus H_1$; then $t+h$ lies in neither subgroup, and the
other case is symmetric.  Hence the least period is $L$.
\end{proof}

\begin{corollary}[Indecomposable split specialization]
If, in addition, $f$ is $k$-indecomposable, then
$d_1=d_2=c_{\mathcal N}=:d$ and
\[
 \mathcal D(f)
 =\{r\geqslant1:d\nmid r\}.
\]
Thus $\operatorname{per}(f)=d$ and $\delta_f=1-1/d$.  If $m>2$, one
may identify $E[\ell]$ with the one-dimensional
$\mathbb F_{\ell^2}$-space $\mathbb F_{\ell^2}$ so that $\beta$ and $F$
act by multiplication by $\rho$ and $\alpha$; then
\(d=\operatorname{ord}_{\mathbb F_{\ell^2}^\times}(\alpha^m).\)
For $m=2$ the same formula holds after identifying an irreducible Frobenius
action with multiplication by $\alpha\in\mathbb F_{\ell^2}^\times$.
\end{corollary}

\begin{proof}
For $m=2$, $\Lambda_{\mathcal N}=\{\pm I\}$ is scalar, so
indecomposability is irreducibility of $F$.  If $\alpha$ is one eigenvalue,
the other is $\alpha^\ell$, and hence
$d_1=d_2=\operatorname{ord}(\alpha^2)=c_{\mathcal N}$, since
$F^\nu\in\{\pm I\}$ exactly when $(\alpha^2)^\nu=1$.

Let $m>2$.  If $\beta$ had an eigenline over $\mathbb F_\ell$, then that
line would also be $F$-stable, contradicting
Proposition~\ref{prop:decomposition}.  Thus $\beta$ is irreducible and
$\mathbb F_\ell[\beta]\cong\mathbb F_{\ell^2}$.  Since $F$ commutes with
$\beta$, identify $E[\ell]$ with $\mathbb F_{\ell^2}$ so that $\beta,F$
act by multiplication by $\rho,\alpha$.  Then
\[
 F^\nu\in\langle\beta\rangle
 \Longleftrightarrow
 (\alpha^m)^\nu=1,
\]
so $c_{\mathcal N}=\operatorname{ord}(\alpha^m)$.  The two $F$-eigenvalues
are $\alpha,\alpha^\ell$, whose $m$th powers have the same order; hence
$d_1=d_2=c_{\mathcal N}$.  The rest follows from
Theorem~\ref{thm:split-spectral-criterion}.
\end{proof}

\subsubsection{Nonsplit rank-two formulas}

Assume now
\(m\in\{3,4,6\}, \qquad q\equiv-1\pmod m.\)
Let $c\in\mathbb F_\ell^\times$ be the residue class of $-q$, and put
\(h:=\operatorname{ord}_{\mathbb F_\ell^\times}(c).\)

\begin{lemma}[Frobenius square in the nonsplit case]
\label{lem:nonsplit-frobenius-square}
On $E[\ell]$ one has
\[
 F^2=cI.
\]
Moreover $\langle\beta,F\rangle$ acts irreducibly on $E[\ell]$.
\end{lemma}

\begin{proof}
Over $\overline{\mathbb F}_\ell$, the relation $F\beta F^{-1}=\beta^{-1}$ exchanges the two
$\beta$-eigenlines.  Hence $\operatorname{tr}(F)=0$, while
$\det(F)=q$ by Lemma~\ref{lem:determinant-eigenvalue-identities}.
Cayley--Hamilton gives $F^2=-qI=cI$.  If $\beta$ has no
$\mathbb F_\ell$-eigenline, there is no common invariant line.  Otherwise
its two eigenlines are defined over $\mathbb F_\ell$, and $F$ exchanges
them.  Thus the full action is irreducible in either case.
\end{proof}

\begin{theorem}[Nonsplit criterion by parity and residue]
\label{thm:nonsplit-parity-residue-criterion}
Under the standing hypotheses of this subsection, assume
$m\in\{3,4,6\}$ and $q\equiv-1\pmod m$.
For every $r\geqslant1$,
\[
 r\in\mathcal D(f)
 \Longleftrightarrow
 \begin{cases}
 h\nmid r,&r\text{ odd},\\[1mm]
 h\nmid mr/2,&r\text{ even}.
 \end{cases}
\]
Equivalently,
\[
 r\in\mathcal D(f)
 \Longleftrightarrow
 \begin{cases}
 \ell\nmid q^r+1,&r\text{ odd},\\[1mm]
 \ell\nmid(-q)^{mr/2}-1,&r\text{ even}.
 \end{cases}
\]
In particular,
\[
 f\text{ permutes }\mathbf P^1(k)
 \Longleftrightarrow
 f\text{ is exceptional over }k
 \Longleftrightarrow
 \ell\nmid q+1.
\]
\end{theorem}

\begin{proof}
Let $r$ be odd.  For each $j$ put $A_j=\beta^{-j}F^r$.  The operator $A_j$
still exchanges the two $\beta$-eigenlines, so
$\operatorname{tr}(A_j)=0$, while $\det(A_j)=q^r$.  Since
$F^r-\beta^j=\beta^j(A_j-I)$ and $\det(\beta)=1$,
\[
 \det(F^r-\beta^j)=\det(A_j-I)=1+q^r,
\]
independently of $j$.  The exact extension-field kernel criterion therefore
reduces to $q^r+1\not\equiv0\pmod\ell$.  Since $r$ is odd, this is
equivalent to $c^r\ne1$, hence to $h\nmid r$.

If $r=2s$, Lemma~\ref{lem:nonsplit-frobenius-square} gives
$F^r=c^sI$.  As $j$ varies, the eigenvalues of $\beta^j$ run through all
$m$th roots of unity.  Hence $c^sI-\beta^j$ is singular for some $j$
exactly when $(c^s)^m=1$, equivalently when $h\mid ms$, or
$h\mid mr/2$.  This proves both formulations and the ground-field criterion.
\end{proof}

\begin{corollary}[Constant field and period]
\label{cor:nonsplit-constant-field-period}
Under the hypotheses of
Theorem~\ref{thm:nonsplit-parity-residue-criterion}, let
\[
 a:=\gcd(h,m),
 \qquad
 \varepsilon_h:=
 \begin{cases}
 1,&h\text{ odd},\\
 0,&h\text{ even}.
 \end{cases}
\]
Then
\[
 \begin{aligned}
 c_{\mathcal N}&=\frac{2h}{\gcd(h,\gcd(m,2))},\\
 \operatorname{per}(f)&=
 \begin{cases}
 2h/a,&h\text{ even},\\[1mm]
 h,&h\text{ odd and }a=1,\\
 2h,&h\text{ odd and }a>1,
 \end{cases}\\
 \delta_f&=1-\frac{a+\varepsilon_h}{2h}.
 \end{aligned}
\]
Thus this least period can be smaller than $c_{\mathcal N}$.
\end{corollary}

\begin{proof}
No odd power of $F$ lies in $\langle\beta\rangle$, since it conjugates
$\beta$ to $\beta^{-1}$ whereas powers of $\beta$ centralize it.  For
$2s$, one has $F^{2s}=c^sI$; the scalar subgroup of $\langle\beta\rangle$
is $\{I\}$ for odd $m$ and $\{\pm I\}$ for even $m$.  Thus the least
positive even power of $F$ lying in $\langle\beta\rangle$ gives
$c_{\mathcal N}=2h/\gcd(h,\gcd(m,2))$.  Writing $a=\gcd(h,m)$, the
even condition $h\mid mr/2$ is equivalent to $2h/a\mid r$; the latter
modulus is even.  Hence the excluded odd degrees are the odd multiples of
$h$, while the excluded even degrees are the multiples of $2h/a$.  Counting
these disjoint progressions gives the displayed formula for $\delta_f$.

It remains to determine the least period.  If $h$ is even, there are no
excluded odd degrees, so the excluded set is exactly the set of multiples of
$2h/a$ and its least period is $2h/a$.  If $h$ is odd and $a=1$, the two
progressions combine to give exactly the multiples of $h$, so the least
period is $h$.

Finally suppose that $h$ is odd and $a>1$.  Since
$m\in\{3,4,6\}$, necessarily $a=3$.  Write $h=3b$; then $b$ is odd.  Modulo
$6b=2h$, the excluded residue classes are $\{0,2b,3b,4b\}$.  Any
stabilizing translation must send $0$ into this set, so it suffices to test
the three nonzero candidates.  Translation by $2b$ sends $3b$ to $5b$,
translation by $3b$ sends $2b$ to $5b$, and translation by $4b$ sends $3b$
to $b$, none of which is excluded.  Hence its least period is $6b=2h$.
\end{proof}

\begin{corollary}[Characteristic-two nonsplit cubic criterion]
Under the standing hypotheses of this subsection, let
$q=2^a$ with $a$ odd and let $m=3$.  Thus
$\mathcal N=E[\ell](\bar k)$ with $\ell\ne2,3$.  Let $h$ be the multiplicative order in $\mathbb F_\ell^\times$ of the
residue class of $-q$.
Then
\[
 r\in\mathcal D(f)
 \Longleftrightarrow
 \begin{cases}
 h\nmid r,&r\text{ odd},\\
 h\nmid3r/2,&r\text{ even},
 \end{cases}
 \qquad
 c_{\mathcal N}=2h.
\]
In particular, the ground-field map is permutation, equivalently
exceptional, exactly when $\ell\nmid q+1$.
\end{corollary}

\begin{proof}
This is Theorem~\ref{thm:nonsplit-parity-residue-criterion} and
Corollary~\ref{cor:nonsplit-constant-field-period} with $m=3$.
\end{proof}

\section{Explicit examples: quotient symmetry and affine translations}
\label{sec:examples}

Examples~\ref{ex:c3} and~\ref{ex:c6} show that changing the linear quotient
symmetry can change the permutation extension degrees while the curve, kernel,
and Frobenius action stay fixed.  The following rank-one remark and
Example~\ref{ex:c4} provide complementary rank-one and nonsplit realizations.
The final example changes the translation component, quotient morphism, and
coordinate without changing the finite-kernel criterion once the induced
linear action is fixed.

\begin{example}[A split $C_3$ quotient]
\label{ex:c3}
Let $E_1/\mathbb F_7:y^2=x^3+1$, and let
$\Lambda=\langle\beta\rangle\cong C_3$ act on $E_1$ by
$\beta(x,y)=(\zeta_3x,y)$.  The quotient coordinate is $y$, and Frobenius
acts trivially on $C_3$ because $7\equiv1\pmod3$.  Take $\varphi=[5]$.
Then $\lvert E_1(\mathbb F_7)\rvert=12$ and $a_7=-4$.  Hence
$a_{7^3}=a_7^3-3\cdot7a_7=20$, so
$\lvert E_1(\mathbb F_{7^3})\rvert=7^3+1-20=324$.
Since $\gcd(5,324)=1$, Proposition~\ref{prop:point-count-multiplication-maps} shows that the
induced degree-$25$ map
\(f^{(3)}_{5,E_1}:E_1/C_3\longrightarrow E_1/C_3\)
permutes $\mathbf P^1(\mathbb F_7)$.

The permutation extension degrees are equally explicit.  On $E_1[5]$, the
Frobenius polynomial is
\(T^2-a_7T+7 \equiv T^2-T+2 \pmod5.\)
It is irreducible over $\mathbb F_5$.  If $\alpha$ is a root in
$\mathbb F_{25}$, then
\(\alpha^{12}=-1, \qquad \alpha^8=1+2\alpha\ne1.\)
Since $\mathbb F_{25}^{\times}$ has order $24$, $\alpha$ has order $24$;
its conjugate has the same order.  The $C_3$ action remains split over every
extension, and Theorem~\ref{thm:cyclic} applied over $\mathbb F_{7^r}$ gives
\[
 r\in\mathcal D(f^{(3)}_{5,E_1})
 \Longleftrightarrow 24\nmid3r
 \Longleftrightarrow 8\nmid r,
 \qquad
 \mathcal D(f^{(3)}_{5,E_1})=\{r\geqslant1:8\nmid r\},
\]
with $\operatorname{per}(f^{(3)}_{5,E_1})=8$ and $\delta_{f^{(3)}_{5,E_1}}=7/8$.
\end{example}

\begin{example}[The same isogeny with a split $C_6$ quotient]
\label{ex:c6}
Keep $E_1/\mathbb F_7:y^2=x^3+1$ and $\varphi=[5]$, but let
$\Lambda=\langle\beta\rangle\cong C_6$ act on $E_1$ by
$\beta(x,y)=(\zeta_3x,-y)$.  The quotient coordinate is
$x^3$, and Frobenius acts trivially on $C_6$ because
$7\equiv1\pmod6$.  Let
\(f^{(6)}_{5,E_1}:E_1/C_6\longrightarrow E_1/C_6\)
be the induced degree-$25$ map.  Its kernel and Frobenius action are exactly
those of Example~\ref{ex:c3}, so the Frobenius eigenvalues on $E_1[5]$ again
have order $24$.  Applying the split $C_6$ criterion over
$\mathbb F_{7^r}$ therefore gives
\[
 r\in\mathcal D(f^{(6)}_{5,E_1})
 \Longleftrightarrow 24\nmid6r
 \Longleftrightarrow 4\nmid r,
 \qquad
 \mathcal D(f^{(6)}_{5,E_1})=\{r\geqslant1:4\nmid r\}.
\]
In particular, $f^{(6)}_{5,E_1}$ permutes
$\mathbf P^1(\mathbb F_7)$, with $\operatorname{per}(f^{(6)}_{5,E_1})=4$ and $\delta_{f^{(6)}_{5,E_1}}=3/4$.

Comparing Examples~\ref{ex:c3} and~\ref{ex:c6}, the curve $E_1$, the
isogeny $[5]$, the kernel $E_1[5]$, and its Frobenius action are unchanged,
whereas the quotient symmetry changes from $C_3$ to $C_6$.  Nevertheless,
the set of permutation extension degrees changes from $\{r:8\nmid r\}$ to
$\{r:4\nmid r\}$, while the least period changes from $8$ to $4$ and $\delta_f$
from $7/8$ to $3/4$.  Thus these extension degrees are not determined by the isogeny kernel and
Frobenius alone; the induced quotient action is essential.
\end{example}

\begin{remark}[Effect of the quotient symmetry order in rank one]
There is an analogous rank-one phenomenon on the same CM curve.  Let
\(E/\mathbb F_7:y^2=x^3+1.\)
Its Frobenius polynomial is
\(T^2+4T+7\equiv(T-4)(T-5)\pmod{13}.\)
Let $\mathcal N\subset E[13]$ be the Frobenius eigenspace with multiplier
$5$.  Since $13\equiv1\pmod6$, this line is stable under both the
order-three and order-six linear symmetries.  The same cyclic $13$-isogeny
kernel therefore gives
\(c_{\mathcal N}^{(3)}=\operatorname{ord}_{\mathbb F_{13}^\times}(5^3)=4, \qquad c_{\mathcal N}^{(6)}=\operatorname{ord}_{\mathbb F_{13}^\times}(5^6)=2.\)
Thus the corresponding common sets of permutation and exceptionality extension
degrees are
\(\{r\geqslant1:4\nmid r\} \qquad\text{and}\qquad \{r\geqslant1:2\nmid r\},\)
respectively.  Hence the dependence on the quotient symmetry order visible in
Examples~\ref{ex:c3} and~\ref{ex:c6} already occurs for a cyclic
prime-degree isogeny kernel.
\end{remark}

\begin{example}[An explicit nonsplit $C_4$ permutation]
\label{ex:c4}
Let $E_2/\mathbb F_7:y^2=x^3+x$, so $j(E_2)=1728$, and let
$\Lambda=\langle\beta\rangle\cong C_4$ act on $E_2$ by
$\beta(x,y)=(-x,\zeta_4y)$.  Take the quotient coordinate
$u=x^2$.  Since $7\equiv3\pmod4$, the $C_4$ action is nonsplit:
Frobenius acts by inversion.  Take $\varphi=[3]$.  A direct point count gives
$\lvert E_2(\mathbb F_7)\rvert=8$ and $a_7=0$, hence
$\lvert E_2(\mathbb F_{7^2})\rvert=64$.  Since $\gcd(3,64)=1$,
Proposition~\ref{prop:point-count-multiplication-maps} shows that the quotient map is a permutation
over $\mathbb F_7$.

The tripling formula is
\[
 x([3]P)
 =
 x(P)
 \left(
 \frac{x(P)^4-6x(P)^2-3}
 {3x(P)^4+6x(P)^2-1}
 \right)^2.
\]
Passing to $u=x^2$ yields the degree-$9$ rational function
\begin{equation}
\label{eq:c4-explicit}
 f^{(4)}_{3,E_2}(u)
 =
 u
 \left(
 \frac{u^2-6u-3}
 {3u^2+6u-1}
 \right)^4
 \in\mathbb F_7(u).
\end{equation}
Direct evaluation on $\mathbf P^1(\mathbb F_7)$ gives
\[
\begin{array}{c|cccccccc}
u&0&1&2&3&4&5&6&\infty\\ \hline
f^{(4)}_{3,E_2}(u)
&0&1&4&3&2&5&6&\infty,
\end{array}
\]
so \eqref{eq:c4-explicit} fixes six points and interchanges $2$ and $4$.
Thus \eqref{eq:c4-explicit} gives a concrete coordinate realization of the
nonsplit quotient criterion.

To determine all permutation extension degrees, note that Frobenius on $E_2[3]$ has
polynomial
\(T^2+7\equiv T^2+1\pmod3,\)
whose roots have order $4$ in $\mathbb F_9^\times$.  If $r$ is odd, the
quotient remains nonsplit and $\mathcal F_{E_2}^{2r}=-1$ on $E_2[3]$, so the
criterion is satisfied.  If $r$ is even, the quotient becomes split and
$\mathcal F_{E_2}^{4r}=1$ on $E_2[3]$, so it fails.  Therefore
\[
 \mathcal D(f^{(4)}_{3,E_2})
 =
 \{r\geqslant1:r\text{ is odd}\},
\]
with $\operatorname{per}(f^{(4)}_{3,E_2})=2$ and $\delta_{f^{(4)}_{3,E_2}}=1/2$.  In particular,
\eqref{eq:c4-explicit} is exceptional and permutes precisely over the
odd-degree extensions of $\mathbb F_7$.
\end{example}

\begin{example}[A non-origin-preserving $C_2$-action over $\mathbb F_5$]
\label{ex:affine-c2}
Let
\(E/\mathbb F_5:\ y^2=x^3+x, \qquad a=(0,0), \qquad \iota_a(P)=a-P, \qquad \varphi=[3].\)
The group $E(\mathbb F_5)=\{\mathrm O_E,(0,0),(2,0),(3,0)\}$ is
$C_2\times C_2$, so $a\notin2E(\mathbb F_5)$.  Hence $\iota_a$ has no
$\mathbb F_5$-rational fixed point and is not $\mathbb F_5$-conjugate to an
origin-preserving involution.  Since $3a=a$, one has
$[3]\iota_a=\iota_a[3]$.

For $P=(x,y)$ the group law gives
\(\iota_a(x,y)=\left(\frac1x,\frac{y}{x^2}\right),\)
so $v:=y/x$ is invariant and $v^2=x+x^{-1}$.  Hence $x$ satisfies
$x^2-v^2x+1=0$, so $[k(E):k(v)]\leqslant2$.  Since the nontrivial
involution $\iota_a$ fixes $v$, one has
$k(v)\subseteq k(E)^{\langle\iota_a\rangle}$ and
$[k(E):k(E)^{\langle\iota_a\rangle}]=2$; therefore equality holds and
$v$ is a quotient coordinate.  Eliminating $x$ from the tripling formulas
yields the induced degree-$9$ rational function
\[
 f_a(v)=
 v\frac{v^8-v^4+2}{3v^8+v^4-1}
 \in\mathbb F_5(v).
\]
For $v\in\mathbb F_5^\times$, $v^4=1$, so $f_a(v)=-v$; moreover
$f_a(0)=0$ and $f_a(\infty)=\infty$.  Thus $f_a$ permutes
$\mathbf P^1(\mathbb F_5)$.

Here $\lvert E(\mathbb F_5)\rvert=4$, so the Frobenius polynomial on $E[3]$ is
\(T^2-2T+5\equiv T^2+T+2\pmod3.\)
Its roots in $\mathbb F_9^\times$ have order $8$.  Since the image of the linear-part homomorphism for
$\langle\iota_a\rangle$ is $C_2$, Theorem~\ref{thm:cyclic} over $\mathbb F_{5^r}$
gives
\[
 \mathcal D(f_a)
 =\{r\geqslant1:4\nmid r\},
 \qquad
 \operatorname{per}(f_a)=4,
 \qquad
 \delta_{f_a}=\frac34.
\]
Thus this affine action is not conjugate over the ground field to an
origin-preserving one, yet its permutation extension degrees are determined by
the same linear $C_2$ action on the kernel.
\end{example}

Together, these examples and the rank-one remark separate the roles of
linear and translation data.  The linear quotient symmetry can change the
permutation extension degrees, whereas the translation component can change the
$k$-rational quotient morphism and coordinate without changing the kernel
criterion once the induced linear action is fixed.

\section{Concluding remarks}

Together with the genus-zero permutation--exceptionality theorem
\cite{FanG0}, Theorems~\ref{thm:intro-genus-at-most-one-permutation-exceptionality}
and~\ref{thm:intro-genus-at-most-one-decomposition-extension-degrees} yield two
conclusions for every separable rational function whose Galois closure has
genus at most one.  Over every finite extension, the function is a permutation
if and only if it is exceptional, with no indecomposability hypothesis.  For
every $k$-decomposition, the common set $\mathcal D(f)$ is the intersection of
the corresponding sets $\mathcal D(f_i)$ for the factors.  Each factor's
Galois-closure curve is realized by a quotient of the original curve and hence
again has genus zero or one, so the genus-zero and $k$-indecomposable genus-one
results combine by repeated application of \cite[Lemma~3.5]{GMS03}.

For every equivariant-isogeny quotient, the finite data
$(\mathcal N,\mathcal F_{\mathcal N},\Lambda_{\mathcal N})$ determine
permutation and exceptionality over every finite extension, the arithmetic and
geometric monodromy permutation groups, the full constant fields after base
change, and the $k$-decomposition classes.  For a degree-greater-than-one
$k$-indecomposable rational function whose normal closure has genus one,
Theorem~\ref{thm:intro-reconstruction} reconstructs such a datum from the
rational function itself.

The examples also distinguish the roles of linear and translation data.  The linear quotient
symmetry can change the permutation extension degrees, whereas translation
components can change the $k$-rational quotient morphism and coordinate without
changing the kernel criterion or monodromy once the induced linear action on
$\mathcal N$ is fixed; see Example~\ref{ex:affine-c2}.

The only companion-manuscript input needed for the genus-at-most-one conclusion
is the genus-zero permutation--exceptionality theorem \cite{FanG0}; the genus-one
reconstruction, the factorization determined by the intrinsic translation
subgroup, the kernel criterion, affine monodromy, decomposition, and
extension-degree theory are proved here.
The companion \cite{FanEERF} treats the tame $k$-indecomposable exceptional
locus and develops the additional fixed-field classification there: affine
translation parameters and branch arithmetic, exact $k$-M\"obius equivalence,
prescribed-field models, occurrence, and exact enumeration.  The structural
results proved here apply to that locus, while those fixed-field classification
and counting results are not used in the proofs here.

Finally, Remark~\ref{rem:finite-procedure} gives a finite procedure for
computing $\mathcal D(f)$, $\operatorname{per}(f)$, and $\delta_f$ from kernel
or determinant tests.

\newcommand{\etalchar}[1]{$^{#1}$}
\providecommand{\bysame}{\leavevmode\hbox to3em{\hrulefill}\thinspace}
\providecommand{\MR}{\relax\ifhmode\unskip\space\fi MR }
\providecommand{\MRhref}[2]{%
  \href{http://www.ams.org/mathscinet-getitem?mr=#1}{#2}
}
\providecommand{\href}[2]{#2}

\end{document}